\documentclass[a4paper]{article}

\usepackage[T1]{fontenc}
\usepackage{amsmath,amssymb,amsthm,mathrsfs}
\usepackage{cite}
\usepackage{mathtools}
\usepackage{enumitem}
\usepackage{graphicx}

\usepackage[pdftex,bookmarksopen=true,bookmarks=true,unicode,setpagesize]{hyperref}
\hypersetup{colorlinks=true,linkcolor=black,citecolor=black}

\usepackage{pgf,tikz}
\usetikzlibrary{arrows,patterns}

\let\equation=\gather
\let\endequation=\endgather

\allowdisplaybreaks[3]
\numberwithin{equation}{section}

\makeatletter
\renewcommand*{\@fnsymbol}[1]{\ensuremath{\ifcase#1\or 1\or 2\or
   3\else\@ctrerr\fi}}
\makeatother

\newtheorem{theorem}{Theorem}[section]

\newtheorem{lemma}[theorem]{Lemma}
\newtheorem{proposition}[theorem]{Proposition}
\theoremstyle{definition}

\theoremstyle{remark}
\newtheorem{remark}[theorem]{Remark}

\newcounter{assum}
\newenvironment{assum}[1][]{\refstepcounter{assum}\equation\tag{\ensuremath{\mathrm{A}\theassum#1}}}{\endequation}

\DeclareMathOperator*{\esssup}{esssup}
\DeclareMathOperator*{\essinf}{essinf}

\renewcommand{\S}{{S^{d-1}}}

\newcommand{\R}{\mathbb{R}}
\newcommand{\X}{{\mathbb{R}^d}}
\newcommand{\N}{\mathbb{N}}
\newcommand{\eps}{\varepsilon}
\newcommand{\la}{\lambda}
\newcommand{\La}{\Lambda}

\newcommand{\kl}{\varkappa_{{\ell}}}
\newcommand{\kn}{\varkappa_{{n\ell}}}
\newcommand{\Tau}{\Upsilon} 
\newcommand{\x}{\mathcal{X}} 
\newcommand{\K}{\mathscr{K}}
\newcommand{\Tauout}{\mathscr{O}}
\newcommand{\Tauin}{\mathscr{C}}
\newcommand{\1}{1\!\!1}
\newcommand{\m}{{\mathfrak{m}}}
\newcommand{\A}{{\mathfrak{a}}}

\newcommand{\dist}{\mathrm{dist}\,} 

\newcommand{\inter}{{\mathrm{int}}}    
\newcommand{\Buc}{C_{ub}(\X)} 
\newcommand{\M}{{\mathcal{M}_\theta}(\R)}
\newcommand{\Utheta}{{C_\theta}}
\newcommand{\Ltheta}{{E^+_\theta}}
\newcommand{\Xinf}{\mathcal{X}_{\infty}}
\newcommand{\tXinf}{\widetilde{\mathcal{X}}_{\infty}}
\newcommand{\xt}{\mathcal{X}_T}

\newcommand{\locun}{\xRightarrow{\,\mathrm{loc}\ }}

\DeclareFontFamily{U}{mathx}{\hyphenchar\font45}
\DeclareFontShape{U}{mathx}{m}{n}{
      <5> <6> <7> <8> <9> <10>
      <10.95> <12> <14.4> <17.28> <20.74> <24.88>
      mathx10
      }{}
\DeclareSymbolFont{mathx}{U}{mathx}{m}{n}
\DeclareFontSubstitution{U}{mathx}{m}{n}
\DeclareMathAccent{\widecheck}{0}{mathx}{"71} 

\title{Doubly nonlocal {F}isher--{KPP} equation: Front~propagation}

\author{Dmitri Finkelshtein\thanks{Department of Mathematics,
Swansea University, Singleton Park, Swansea SA2 8PP, U.K. ({\tt d.l.finkelshtein@swansea.ac.uk}).} \and Yuri Kondratiev\thanks{Fakult\"{a}t
f\"{u}r Mathematik, Universit\"{a}t Bielefeld, Postfach 110 131, 33501 Bielefeld,
Germany ({\tt kondrat@math.uni-bielefeld.de}).}  \and Pasha Tkachov\thanks{Gran Sasso Science Institute, Viale Francesco Crispi, 7, 67100 L'Aquila AQ, Italy ({\tt pasha.tkachov@gssi.it}).}}

\begin{document}
\maketitle

\begin{abstract}
We study propagation over $\X$ of the solution to a nonlocal nonlinear equation with anisotropic kernels, which can be interpretted as a doubly nonlocal reaction-diffusion  equation of the Fisher--KPP-type. 
We prove that if the kernel of the nonlocal diffusion is exponentially integrable in a direction and if the initial condition decays in this direction faster than any exponential function, then the solution propagates at most linearly in time in that direction. Moreover, if both the kernel and the initial condition have the above properties in any direction (being, in general, anisotropic), then we prove linear in time propagation of the corresponding solution over~$\X$. 

\textbf{Keywords:} nonlocal diffusion, Fisher--KPP equation, nonlocal nonlinearity, long-time behavior, front propagation, anisotropic kernels, integral equation

\textbf{2010 Mathematics Subject Classification:} 35K55, 35K57, 35B40 

\end{abstract}

\section{Introduction} 

We will study front propagation of solutions to the equation
\begin{equation}
\begin{aligned}
\dfrac{\partial u}{\partial t}(x,t)&=\varkappa^+ \int_{\X }a^+  (x-y)u(y,t)dy-m u(x,t)  -u(x,t)\, G\bigl(u(x,t)\bigr),\\
G\bigl(u(x,t)\bigr)&:=\kl u(x,t) + \kn \int_{\X }a^- (x-y)u(y,t)dy.
\end{aligned}
\label{eq:basic}
\end{equation}
Here $d\in\N$; $\varkappa^+, m>0$ and $\kl, \kn\geq0$  are constants, such that 
\begin{equation}\label{eq:nondegencomp}
      \varkappa^- := \kl+\kn > 0;
\end{equation}    
the kernels  $0\leq a^\pm \in L^{1}(\X)$ are probability densities, i.e. $\int_{\X }a^\pm (y)dy=1$.

For the case of the local nonlinearity in \eqref{eq:basic}, when $\kn=0$, the equation \eqref{eq:basic}  was considered, in particular, in \cite{CD2007,Yag2009,AGT2012,Gar2011,LSW2010,SLW2011,CDM2008,Sch1980,ZLW2012,SZ2010}. For a nonlocal nonlinearity and, especially, for the case $\kl=0$ in \eqref{eq:basic}, see e.g. \cite{Dur1988,FM2004,FKK2011a,FKKozK2014,PS2005,FKMT2017,FT2017c,YY2013}. For details, see the introduction to \cite{FKT100-1} and also the comments below.

The present paper is a continuation of \cite{FKT100-1} and \cite{FKT100-2}; they all are based on our unpublished preprint \cite{FKT2015} and thesis \cite{Tka2017}.

By a solution to \eqref{eq:basic} on $[0,T)$, $T\leq \infty$, we will understand the so-called classical solution, that is a continuous mapping from $[0,T)$ to the space $E:=L^{\infty}(\X)$ which is continuously differentiable (in the sense of the $\esssup$-norm in~$E$) in $t\in(0,T)$, and satisfies \eqref{eq:basic}. We denote by $\x_\infty$ the vector space of all continuous mappings from $\R_+$ to $E$. 

By~\cite[Theorem 2.2]{FT2017a}, for any $0\leq u_0\in E$ and for any $T>0$, there exists a unique classical solution $u$ to \eqref{eq:basic} on $[0,T)$. In particular, $u\in\x_\infty$ is a unique classical solution to \eqref{eq:basic} on $\R_+:=[0,\infty)$. 

Moreover, by \cite{FT2017a}, if $u_0$ belongs to either of spaces $C_b(\X)$ or $\Buc$ of bounded continuous or, respectively, bounded uniformly continuous functions on $\X$ with $\sup$-norm, then $u(\cdot,t)$ belongs to the same space for all $t>0$; cf.~\ref{eq:QBtheta_subset_Btheta} in Theorem~\ref{thm:Qholds} below.

We will assume in the sequel, that
\begin{assum}\label{as:beta}
\varkappa^+ >m.
\end{assum}
Under \eqref{as:beta}, the equation \eqref{eq:basic} has two constant stationary solutions: $u\equiv0$ and $u\equiv\theta$, where
\begin{equation}\label{theta_def}
  \theta:=\frac{\varkappa^+ -m}{\varkappa^- }>0.
\end{equation}
Moreover, one can then also rewrite the equation in a reaction-diffusion form
\[
  \dfrac{\partial u}{\partial t}(x,t)=\varkappa^+ \int_{\X } a^+  (x-y)\bigl( u(y,t)-u(x,t)\bigr) dy+u(x,t)\Bigl(\beta - G\bigl(u(x,t)\bigr)\Bigr),
\]
where $\beta=\varkappa^+-m>0$. We treat then \eqref{eq:basic} as a doubly nonlocal Fisher--KPP equation, see the introduction to \cite{FKT100-1} for details.

By \cite[Theorem 1.5, Remark~2.6]{FKT100-1}, the assumption
\begin{assum}\label{as:compar}
  \varkappa^{+}a^{+}(x)\geq \kn \theta a^{-}(x),\quad \text{a.a.}\ x\in\X
\end{assum}
is necessary and sufficient to have that the solution $u(\cdot,t)$ to \eqref{eq:basic} remains in the tube 
\begin{equation}\label{eq:tube}
      E^+_\theta:=\{u\in E\mid 0\leq u\leq\theta\}
\end{equation}
for all positive times $t>0$, given that $u_0\in E^+_\theta$. Here and in the sequel, we will understand all inequalities between functions from $E$ almost everywhere only.

Note that the assumption \eqref{as:compar} is redundant for the case of the local nonlinear part in \eqref{eq:basic}, i.e.~where $\kn=0$. The assumptions \eqref{as:beta}--\eqref{as:compar} ensure the comparison principle for the equation \eqref{eq:basic}, see Proposition~\ref{prop:fullcomp} below.

Through the paper we will assume also that
\begin{assum}\label{as:bdd}
    a^+\in L^\infty(\X).
  \end{assum}
Clearly, for the case $\kn>0$, \eqref{as:compar}--\eqref{as:bdd} imply $a^-\in L^\infty(\X)$.

Let $\S$ denote the unit sphere in $\X$ centered at the origin.
For a fixed $\xi\in\S $, we assume that
  \begin{assum}[_\xi]\label{as:firstmoment}
   \int_\X \lvert x\cdot \xi \rvert \, a^+(x)\,dx<\infty.
  \end{assum}
Here and below $x\cdot \xi$ denotes the scalar product in $\X$.
Under the assumption \eqref{as:firstmoment}, we define
\begin{equation}\label{firstdirmoment}
\m_\xi:=\varkappa^+ \int_\X x\cdot \xi \, \, a^+(x)\,dx.
\end{equation}

For the fixed $\xi\in\S $, we assume also, that
\begin{assum}[_\xi]\label{as:nondegdir}
    \begin{gathered}
    \text{there exist $r=r(\xi)\geq0$, $\rho=\rho(\xi)>0$, $\delta=\delta(\xi)>0$, such that}\\
    a^+(x)\geq\rho, \text{ for a.a.\! $x\in B_{\delta}(r\xi)$.}
    \end{gathered}
  \end{assum}
Here and below $B_\rho(y)$ denotes the ball in $\X$ of the radius $\rho>0$ centered at the point $y\in\X$.

For an arbitrary direction $\xi\in\S$, we define
\begin{equation}\label{aplusexpla}
  \A_\xi(\la):=\int_\X a^+(x) e^{\la x\cdot \xi}\,dx\in(0,\infty], \quad \la>0.
\end{equation}
We assume that, for the fixed $\xi\in\S$,
\begin{assum}[_\xi]\label{as:expintdir}
  \text{there exists} \ \mu=\mu(\xi)>0 \ \text{such that} \ \A_{\xi}(\mu)<\infty.
\end{assum}
Under condition \eqref{as:expintdir}, we consider, see \cite{FKT100-2} for details,
\begin{equation}\label{eq:abskernel}
      \sigma_\xi(a^+):=\sup\bigl\{\la>0\bigm\vert \A_{\xi}(\la)<\infty\bigr\}\in(0,\infty].
\end{equation}

The front propagation in a direction $\xi\in\S$ is deeply related to the minimal speed of traveling wave solutions in the direction $\xi$. By a (monotone) traveling wave solution to \eqref{eq:basic} in the fixed direction $\xi\in \S $, we will understand a solution of the form 
\begin{equation}\label{eq:deftrw}
      \begin{gathered}
      u(x,t)=\psi(x\cdot\xi-ct),  \quad t\geq0, \ \mathrm{a.a.}\ x\in\X, \\
      \psi(-\infty)=\theta, \qquad \psi(+\infty)=0,
      \end{gathered}
\end{equation}
where $c\in\R$ is called the speed of the wave and a decreasing and~right-continuous function $\psi$ is called the profile of the wave.

To formulate the main result, one needs the following Theorem, which we have proved in \cite{FKT100-2}.

\begin{theorem}[{{cf.~\cite[Theorems 1.1--1.3]{FKT100-2}}}]\label{thm:trwall}
Let \eqref{as:beta}--\eqref{as:compar} hold, and, for a fixed $\xi\in\S$, let \eqref{as:expintdir} hold. 
\begin{enumerate}
  \item Then there exists $c_*(\xi)\in\R$, such that, for any $c\geq c_*(\xi)$, there exists a profile $\psi=\psi_c$, such that \eqref{eq:deftrw} defines a solution to \eqref{eq:basic}; and for any $c<c_*(\xi)$ a traveling wave solution to \eqref{eq:basic} of the form \eqref{eq:deftrw} does not exist.
  \item Let, additionally, \eqref{as:bdd} and \eqref{as:firstmoment}--\eqref{as:nondegdir} hold. Then here exists a unique 
\begin{equation}\label{eq:lastar}
      \la_*=\la_*(\xi)\in(0,\infty), \qquad  \la_*(\xi)\leq \sigma_\xi(a^+),
\end{equation}
such that
\begin{equation}\label{eq:cstar}
    c_*(\xi)=\min_{\la>0}\frac{\varkappa^+\A_\xi(\la)-m}{\la}=\frac{\varkappa^+\A_\xi(\la_*)-m}{\la_*}> \m_\xi.
\end{equation}
Moreover, the abscissa of a profile $\psi_{*,\xi}$ corresponding to the traveling wave with the minimal speed $c_*(\xi)$ coincides with $\la_*(\xi)$, namely,
\begin{equation}\label{eq:absofprofile}
    \sup\biggl\{ \la>0 \biggm\vert \int_\R\psi_{*,\xi}(s)e^{\la s}\,ds<\infty\biggr\}=\la_*(\xi).
\end{equation}
\end{enumerate}
\end{theorem}
Note also that, under some additional technical assumptions, see \cite[Theorem~1.3]{FKT100-2}, the profile $\psi_c$ corresponding to a speed $c\geq c_*(\xi)$, $c\neq0$ is unique (up to a shift).

We start with a result that, for a solution $u(x,t)$ to \eqref{eq:basic}, the function $u(tx,t)$ converges (when $t\to\infty$) to $0$ uniformly on the hyperspace $\{x\cdot\xi\geq (1+\eps)c_*(\xi)\}$ for each $\eps>0$.

\begin{theorem}\label{thm:decayoutsidedirectional}
 Let \eqref{as:beta}--\eqref{as:bdd} hold. For a fixed $\xi\in\S$, suppose also that \eqref{as:firstmoment}--\eqref{as:expintdir} hold. Let $\la_*=\la_*(\xi)>0$ be the same as in \eqref{eq:lastar}--\eqref{eq:cstar}. Let $u_0\in E_\theta^+$ be such that 
\begin{equation}\label{eq:fastinitcond}
\lVert u_0\rVert_{\la_*,\xi}:=\esssup_{x\in\X} u_0(x) e^{\la_*  x\cdot\xi }<\infty.
\end{equation}
Let $u\in\x_\infty$ be the corresponding classical solution to \eqref{eq:basic} on $\R_+$. Let $\Tauout_\xi\subset\X$ be an open set, such that 
\begin{equation}\label{eq:condonfrontdir}
    \Tau_*(\xi):=\bigl\{ x\in\X \mid x\cdot\xi\leq c_*(\xi)\bigr\}\subset \Tauout_\xi,
\end{equation}
and  $\delta:=\dist (\Tau_*(\xi),\X\setminus\Tauout_\xi )>0$. Then the following estimate holds
\begin{equation}\label{supconvto0xi}
  \esssup_{x\notin t\Tauout_\xi } u(x,t)\leq \lVert u_0\rVert_{\la_*,\xi} e^{-\la_*\delta t}, \quad t>0.
\end{equation}
\end{theorem}
Here and below $tA:=\{tx\mid x\in A\}$ for a measurable $A\subset\X$.

Now we are going to formulate our results about the front propagation in different directions.
If, for a $\xi\in\S$, the assumption \eqref{as:expintdir} fails, i.e. if $\A_{\xi}(\la)=\infty$ for all $\la>0$, we will set $c_*(\xi):=\infty$. Consider the set
\begin{equation}\label{eq:condonfrontglobal}
    \Tau_*:=\bigcap_{\xi\in\S} \Tau_*(\xi)=\bigl\{ x\in\X \mid x\cdot\xi\leq c_*(\xi), \ \xi\in\S \bigr\}.
\end{equation}
Clearly, $\Tau_*$ is a closed convex subset of $\X$. In particular, if \eqref{as:expintdir} fails for all $\xi\in\S$, then $\Tau_*=\X$.

We assume the following modification of \eqref{as:firstmoment}:
\begin{equation}\label{firstglobalmoment}
  \int_\X |x| a^+(x)\,dx<\infty. \tag{A4}
\end{equation}
Clearly, \eqref{firstglobalmoment} yields that  \eqref{as:firstmoment} holds for all $\xi\in\S$. 
Under \eqref{firstglobalmoment}, we define, cf.~\eqref{firstdirmoment},
\begin{equation}\label{firstfullmoment}
  \m:=\varkappa^+ \int_\X x a^+(x)\,dx\in\X.
\end{equation} 

Next, we assume that
\begin{equation}\label{as:nondegglobalmod}\tag{A5}
  \begin{gathered}
    \text{there exists $\rho,\delta>0$, such that} \\
    \varkappa^{+}a^{+}(x)-\kn \theta a^{-}(x)\geq\rho, \text{ for a.a. }  x\in B_\delta(0).
  \end{gathered}
\end{equation}
Clearly, \eqref{as:nondegglobalmod} implies that
\begin{equation} \label{as:nondegglobal}
\text{there exists $\rho,\delta>0$ such that} \ a^{+}(x)\geq\rho, \text{ for a.a. } x\in B_\delta(0),
\end{equation}
in particular, then, for all $\xi\in\S$, \eqref{as:nondegdir} holds with $r(\xi)=0$. 

Note that, see Proposition~\ref{prop:verynew} below, if \eqref{as:beta}--\eqref{as:nondegglobalmod} hold and if, for some $\xi\in\S$, \eqref{as:expintdir} holds, then
\begin{equation}\label{eq:momentbelongsfront}
      \m\in\inter(\Tau_*).
\end{equation}
Here and below, for a closed set $A\subset\X$, we denote by $\inter(A)$ the interior of $A$.

The following theorem states, informally, that, for a solution $u(x,t)$ to \eqref{eq:basic} and for any $\eps>0$, the function $u(tx,t)$ converges (as $t\to\infty$) to $\theta$ locally uniformly on the set $(1-\eps)\Tau_*$ and converges to $0$ locally uniformly out of the set $(1+\eps)\Tau_*$. Moreover, if $\Tau_*$ is bounded, then the latter convergence holds uniformly.
\begin{theorem}\label{thm:combi}
Let the conditions \eqref{as:beta}--\eqref{as:nondegglobalmod} hold. Let $u_0\in E_\theta^+$ and $u\in\x_\infty$ be the corresponding classical solution to \eqref{eq:basic} on $\R_+$.
\begin{enumerate}
  \item Let there exist $\xi\in\S$, such that \eqref{as:expintdir} holds. Let $u_0$ be such that \eqref{eq:fastinitcond} holds for all those $\xi\in\S$ where $c_*(\xi)<\infty$. Then, for any compact set $\Tauin\subset \X\setminus\Tau_*$, there exist $\nu=\nu(\Tauin)>0$ and $D=D(u_0,\Tauin)>0$, such that
\begin{equation}\label{eq:globalaboveresult}
  \esssup_{x\in t\Tauin} u(x,t)\leq D  e^{-\nu t}, \quad t>0.
\end{equation}
\item If, additionally, the set $\Tau_*$ is bounded (and hence compact), then \eqref{eq:globalaboveresult} holds for any (unbounded) closed set $\Tauin\subset \X\setminus\Tau_*$.
\item Let $u_0$ be such that there exist $x_0\in\X$, $\eta>0$, $r>0$, with $u_0(x)\geq\eta$ for a.a.~$x\in B_r(x_0)$. Then, for any compact set $\Tauin\subset\inter(\Tau_*)$,
  \begin{equation}\label{eq:globalbelowresult}
    \lim_{t\to\infty} \essinf_{x\in t\Tauin} u(x,t)=\theta.
  \end{equation}
\end{enumerate}
\end{theorem}

By Proposition~\ref{prop:front_is_non-empty} below, a sufficient condition that $\Tau_*$ is a compact set is that
\begin{gather}\label{as:expintglobal}
  \text{there exists $\mu_d>0$ such that} \ \int_\X a^+(x) e^{\mu_d|x|}\,dx<\infty. \tag{A6}
\end{gather}
Evidently, \eqref{as:expintglobal} implies \eqref{firstglobalmoment}. We will show in Remark~\ref{rem:justequiv}, that \eqref{as:expintglobal} is equivalent to that \eqref{as:expintdir} holds for all $\xi\in\S$ or just for all $\xi\in\{e_i,-e_i\mid 1\leq i\leq d\}$ with an arbitrary orthonormal basic $\{e_i\mid 1\leq i\leq d\}$ in $\X$. 

Note also that, for the first two items of Theorem~\ref{thm:combi}, it is enough to assume \eqref{as:nondegglobal} instead of \eqref{as:nondegglobalmod}.

By the mentioned above, if $u_0$ belongs to $C_b(\X)$ or $\Buc$, then $u(\cdot,t)$ is continuous for $t>0$, and one can replace $\esssup$/$\essinf$ in \eqref{eq:globalaboveresult}, \eqref{eq:globalbelowresult} by $\max$/$\min$, correspondingly (note that in the third item of Theorem~\ref{thm:combi} we shall assume then that $u_0\not\equiv0$).

On Figure~\ref{fig:f1}, we sketched a relation between
$\Tau_*(\xi)$ and $\Tau_*$. The arrows describe the `motion' of the sets $t\Tau_*(\xi)$ and $t\Tau_*$, correspondingly. By Proposition~\ref{prop:front_is_non-empty} below, $\varkappa^+\m\in\Tau_*$, however, the origin may be out of $\Tau_{*,\xi}$, for some $\xi\in\S$, and hence out of $\Tau_*$. However, by Remark~\ref{rem:minspeedopospos} below, for each $\xi\in\S$, the origin must belong to at least one of the sets $\Tau_{*}(\xi)$ and $\Tau_{*}(-\xi)$. Note also that $0\in\inter{\Tau_*}$ (which does hold, if e.g. $a^+(-x)=a^+(x)$, $x\in\X$, then $\m=0$) implies, by \eqref{eq:cstar}, that all traveling waves move to their `right directions', cf. Remark~\ref{rem:zero_is_in_front} below.

\begin{figure}[!ht]
  \centering
   \begin{tikzpicture}[xscale=0.4,yscale=0.4,line width=1pt,>=stealth]
\draw[pattern=north west lines, pattern color=gray!80!white] plot [smooth cycle,tension=0.4] coordinates {(2,2) (1,4) (2,7) (3,8) (5,9) (9,9) (12,7) (13,6) (14,4) (13,2) (12,1) (11,0.7) (10,0.6) (9,0.5) (8,0.6) (7,0.7) (6,0.8) (5,0.9) (4,1) (3,1.2)};
\node (T) at (9,6) {$\Tau_*$};
\fill[black] (5,3) circle (1.5ex) node[right] {$\varkappa^+\m$};
\draw (19,0) -- (5,16);
\fill[pattern=north east lines, pattern color=gray!20!white] (19,0) -- (5,16) -- (0,16) -- (0,-2) -- (19,-2) -- cycle;
\node (T1) at (7,11) {$\Tau_{*}(\xi)$};
\draw[->] (12.5,-2) node[below] {$O$} -- (17.1,2.2) node[right] {$c_*(\xi)\xi$};
\fill[black] (12.5,-2)  circle (1.5ex);
\draw[->] (12.5,-2) -- (14.5,{4.2*2/4.6-2}) node[below] {$\xi$};
\foreach \x in {6,...,16}
    \draw[->] (\x,{-8*\x/7+8*19/7}) -- ({\x+1}, {7*(\x+1)/8+(-8*\x/7+8*19/7)-7*\x/8});
\node[right,align=center] (F) at ({12+1.5}, {7*(12+1)/8+(-8*12/7+8*19/7)-7*12/8}) {front  propagation\\ in a direction $\xi$};
\draw[->] (9,9)--(9,10);
\draw[->] (5,9)--(5,10);
\draw[->] (11,0.7)--(11,-0.3);
\draw[->] (8,0.6)--(8,-0.4);
\draw[->] (5,0.9)--(5,-0.1);
\draw[->] (1,4)--(0,4);
\draw[->] (13,2)--(14,1.2);
\draw[->] (2,2)--(1,1);
\draw[->] (2,7)--(0.8,7.6);
\node[below,align=center] (H) at (7,-0.1) {front  propagation};
\end{tikzpicture}
  \caption{Relationship between the sets $\Tau_{*}(\xi)$ and $\Tau_*$}
  \label{fig:f1}
\end{figure}
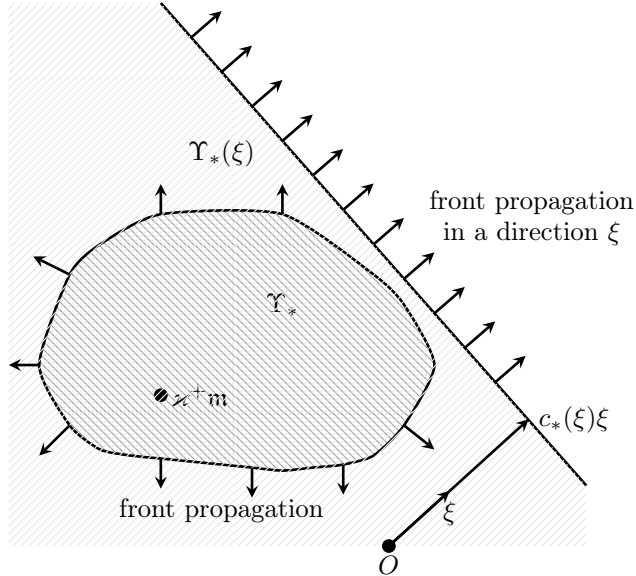

The notion of `front' has several slightly different definitions, see e.g. \cite{PS2005,Xin2009}. Informally, front for \eqref{eq:basic} has to be a set which separates $\Tauin\subset\X$, where $u(tx,t)\to\theta$, $x\in\Tauin$, $t\to\infty$, and $\Tauout\subset\X$, where $u(tx,t)\to0$, $x\in\Tauout$, $t\to\infty$. The results of Theorem~\ref{thm:combi} show that any $\eps$-neighborhood of the boundary of $\Tau_*$ can be considered as a front set in the meaning above.
Figures~\ref{fig:f2}, \ref{fig:f3} describe two `projections' of the three-dimensional graph for $u=u(x,t)$.

\begin{figure}[!ht]
  \centering
  \begin{tikzpicture}[xscale=0.55,yscale=0.43,line width=1pt,every node/.style={font=\small},declare function={u(\x)=(8-0.4)*exp(-0.025*\x*\x);},>=stealth]
  \pgfmathsetmacro{\bb}{10};
  \pgfmathsetmacro{\th}{8};
  \draw[dashed, line width=0.7pt,->] ({-(\bb+0.5)},0)--({\bb+1},0) node[below] {$\xi$};
  \draw[dashed, line width=0.7pt,->] (-1,-1)--(-1,{\th+1.5}) node[left] {$u$};
  \draw[domain=-10:10,smooth,variable=\x,black,line width=1.5pt] plot (\x,{u(\x)});
  \draw({-(\bb+0.5)},\th)--({\bb+0.5},\th);
  \node[above left] (T) at (-1,\th) {$\theta$};
  \pgfmathsetmacro{\cxi}{\bb-3};
  \pgfmathsetmacro{\bbm}{-\bb+1};
  \pgfmathsetmacro{\cxim}{-\bb+4};
  \pgfmathsetmacro{\cximmm}{-\bb+4};
  \foreach \x in {\cxi,...,\bb}{
    \draw[->,line width=0.5pt] (\x,{u(\x)})--(\x,0.1);
    \draw[->,line width=0.5pt] ({\x-0.5},{u(\x-0.5)})--({\x-0.5},0.1);}
  \foreach \x in {\bbm,...,\cximmm}{
    \draw[->,line width=0.5pt] ({\x},{u(\x)})--({\x},0.1);
    \draw[->,line width=0.5pt] ({\x-0.5},{u(\x-0.5)})--({\x-0.5},0.1);}
    \pgfmathsetmacro{\cximm}{\cxim+2.5};
    \pgfmathsetmacro{\cxin}{\cxi-3}
  \foreach \x in {\cximm,...,\cxin}{
    \draw[->,line width=0.5pt] (\x,{u(\x)})--(\x,{\th-0.1});
    \draw[->,line width=0.5pt] ({\x+0.5},{u(\x+0.5)})--({\x+0.5},{\th-0.1});
  }
  \fill[black] ({\cxi-1.5},0) circle(0.7ex) node[below] {$t c_*(\xi)\xi$};
  \fill[black] ({\cxim+1.3},0) circle(0.7ex) node[below] {$-t c_*(-\xi)\xi$};
  \draw[dashed,line width=0.7pt] ({\cxi-1.5},0)--({\cxi-1.5},\th);
  \draw[dashed,line width=0.7pt] ({\cxim+1.3},0)--({\cxim+1.3},\th);
  \draw[dashed,line width=0.7pt,->] ({\cxim+1.3},{u(\cxim+1.3)}) --({\cxim+2.1},{u(\cxim+1.3)}) node[right] {$\eps t\xi$};
  \draw[dashed,line width=0.7pt,->] ({\cxim+1.3},{u(\cxim+1.3)}) --({\cxim+0.6},{u(\cxim+1.3)}) node[left] {$-\eps t\xi$};
  \draw[dashed,line width=0.7pt,->] ({\cxi-1.5},{u(\cxi-1.5)}) --({\cxi-0.8},{u(\cxi-1.5)}) node[right] {$\eps t\xi$};
  \draw[dashed,line width=0.7pt,->] ({\cxi-1.5},{u(\cxi-1.5)}) --({\cxi-2.2},{u(\cxi-1.5)}) node[left] {$-\eps t\xi$};
  \draw[dashed,line width=0.7pt] ({\cxim+2.1},0)--({\cxim+2.1},\th);
  \draw[dashed,line width=0.7pt] ({\cxim+0.6},0)--({\cxim+0.6},\th);
  \draw[dashed,line width=0.7pt] ({\cxi-0.8},0)--({\cxi-0.8},\th);
  \draw[dashed,line width=0.7pt] ({\cxi-2.2},0)--({\cxi-2.2},\th);
      \end{tikzpicture}
 \caption{Space-value diagram}
  \label{fig:f2}
\end{figure}
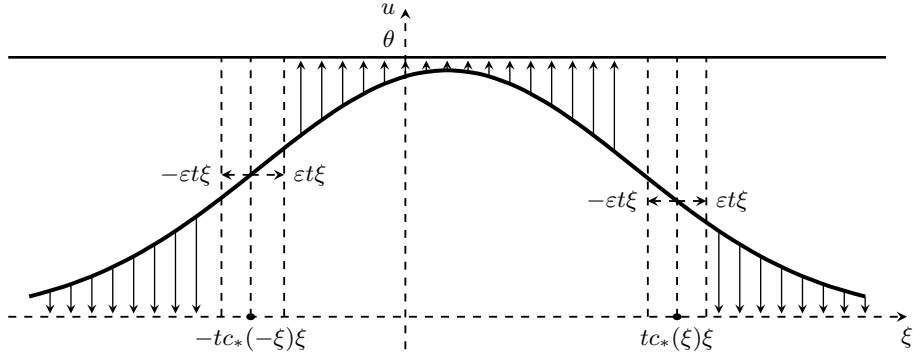

\begin{figure}[ht!]
  \centering
  \begin{tikzpicture}[xscale=1,yscale=1,line width=1pt,every node/.style={font=\small},>=stealth]
  \pgfmathsetmacro{\r}{8};
  \pgfmathsetmacro{\p}{7};
  \pgfmathsetmacro{\pp}{-1.5};
  \pgfmathsetmacro{\ra}{0.2};
  \pgfmathsetmacro{\rb}{0.31};
  \pgfmathsetmacro{\rc}{0.6};
  \pgfmathsetmacro{\rd}{0.72};
  \coordinate (O) at (0,0);
  \coordinate (T) at ({\r*1.1},0);
  \coordinate (X) at (0,\p);
  \draw[->,line width=0.7pt] (0,\pp) -- (X);
  \draw[->,line width=0.7pt] ({-\r*0.07},0) -- (T);
      \coordinate (A) at (\r,{\r*\ra});
      \coordinate (B) at (\r,{\r*\rb});
      \coordinate (C) at (\r,{\r*\rc});
      \coordinate (D) at (\r,{\r*\rd});
      \coordinate (E1) at (\r,{0.9*\p});
      \coordinate (E2) at (0,{0.9*\p});
      \coordinate (F1) at (\r,{0.9*\pp});
      \coordinate (F2) at (0,{0.9*\pp});
      \pgfmathsetmacro{\q}{7};
      \coordinate (G) at (\q,0);
      \coordinate (GA) at (\q,{\q*\ra});
      \coordinate (GB) at (\q,{\q*\rb});
      \coordinate (GC) at (\q,{\q*\rc});
      \coordinate (GD) at (\q,{\q*\rd});
      \coordinate (GAO) at (0,{\q*\ra});
      \coordinate (GBO) at (0,{\q*\rb});
      \coordinate (GCO) at (0,{\q*\rc});
      \coordinate (GDO) at (0,{\q*\rd});
      \draw [dashed] (G) -- (GD) -- (GDO);
      \draw [dashed] (GA) -- (GAO);
      \draw [dashed] (GB) -- (GBO);
      \draw [dashed] (GC) -- (GCO);
      \draw[pattern=north west lines, pattern color=gray!40!white] (0.1,{0.04*(\rb+\rc)})--(B)--(C)--cycle;
      \draw[pattern=dots, pattern color=black!80!gray] (D)--(E1)--(E2)--(O)--cycle;
      \draw[pattern=dots, pattern color=black!80!gray] (A)--(F1)--(F2)--(O)--cycle;
      \node [left] at (GAO) {$t(c_*(-\xi)-\varepsilon)$};
      \node [left] at (GBO) {$t(c_*(-\xi)+\varepsilon)$};
      \node [left] at (GCO) {$t(c_*(\xi)-\varepsilon)$};
      \node [left] at (GDO) {$t(c_*(\xi)+\varepsilon)$};
      \node [left] at (X) {$S^{d-1}\ni\xi$};
       \node [below] at (T) {$\mathbb{R}_+$};
        \node [below] at (G) {$t$};
        \draw (O) -- (\r,{\r*(\ra + \rb)/2}) node [right] {$x=t c_*(-\xi)\xi$};
        \draw (O) -- (\r,{\r*(\rc + \rd)/2}) node [right] {$x=t c_*(\xi)\xi$};
      \coordinate (K1) at ({0.3*\r},{0.48*\r*(\rb+\rc)});
      \node [below] at (K1) {$u(x,t)\to 0$};
      \coordinate (K2) at ({0.5*\r},{0.5*\pp});
      \node at (K2) {$u(x,t)\to 0$};
      \coordinate (K3) at ({0.58*\r},{0.36*\r*(\rb+\rc)});
      \node[right] at (K3) {$u(x,t)\to \theta$};
      \coordinate (t1) at (\q,{\q*(\ra + \rb)/2});
      \coordinate (t2) at (\q,{\q*(\rc + \rd)/2});
    \coordinate (t3) at (\q,{\q*(\ra + \rb+\rc + \rd)/4});
      \draw[line width=2.5pt,color=black!80!white] (t1) -- (t2);
    \node[below,rotate=90] at (t3) {$t\Tau_*$};
      \pgfmathsetmacro{\u}{{0.6*\q}};
     \node[below] (U) at (\u,0) {$1$};
    \draw[dashed] (U) -- (\u,{\u*\rd});
    \draw[line width=2.5pt,color=black!80!white] (\u,{\u*(\ra+\rb)/2}) -- (\u,{\u*(\rc+\rd)/2});
    \node[above,rotate=90] at (\u,{0.46*\u*(\rb+\rc)}) {$\Tau_*$};
      \end{tikzpicture}
 \caption{Space-time diagram}
  \label{fig:f3}
\end{figure}
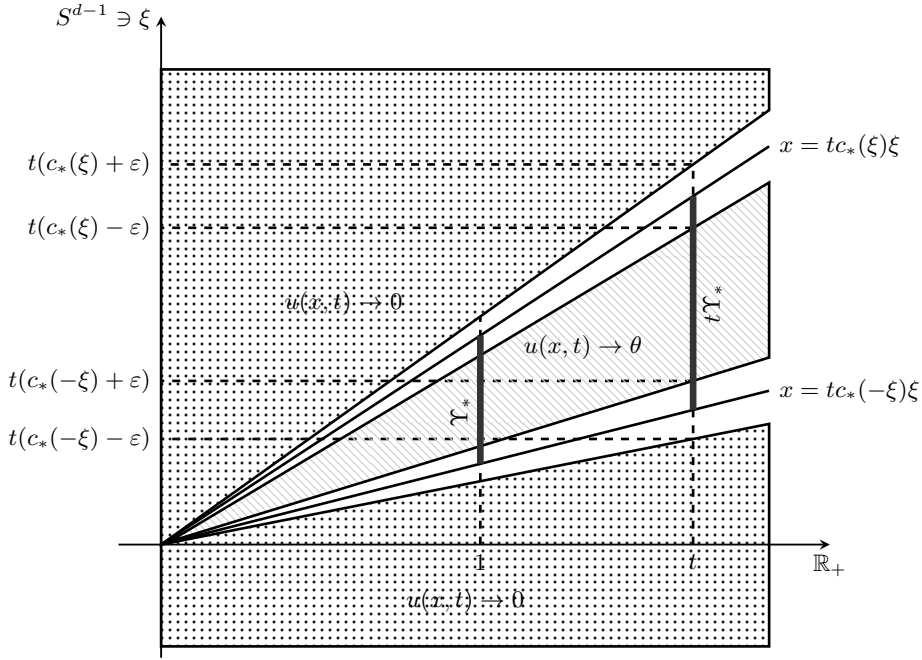

For the case $\kl=0$, $\varkappa^+=\varkappa^-=\kn$, $a^+(x)=a^-(x)$ for $x\in\X$, the result similar to Theorem~\ref{thm:combi} can be found in \cite{PS2005}, where the viscosity solution technique has been used. If, additionally, $d=1$ and the kernels $a^\pm$ decay faster than any exponential function, one can refer also to \cite{WZ2006}. For the case $\kn=0$, see also \cite{SZ2010}.

Recall that  if the condition \eqref{as:expintdir} fails for all $\xi\in\S$, then $\Tau_*=\X$, and hence \eqref{eq:globalaboveresult} has no sense. However, the third item of Theorem~\ref{thm:combi} is valid, and \eqref{eq:globalbelowresult} implies the so-called infinite speed of propagation. Namely, the following 
statement holds.
\begin{proposition}\label{prop:infinitespeed}
 Let the conditions \eqref{as:beta}--\eqref{as:nondegglobalmod} hold. Suppose that
 \begin{equation}\label{eq:heavy}
       \int_\X a^+(x) e^{\la x\cdot \xi}\,dx=\infty, \quad \la>0, \  \xi\in\S.
  \end{equation} 
 \begin{enumerate}
   \item Let $u_0\in E_\theta^+$ be such that there exist $x_0\in\X$, $\eta>0$, $r>0$, with $u_0(x)\geq\eta$ for a.a.~$x\in B_r(x_0)$, and let $u\in\x_\infty$ be the corresponding classical solution to \eqref{eq:basic} on $\R_+$. 
 Then, for any compact $\Tauin\subset\X$, the convergence \eqref{eq:globalbelowresult} holds.
  \item There does not exist a traveling wave solution of the form \eqref{eq:deftrw} to the equation~\eqref{eq:basic}
 \end{enumerate}
 \end{proposition} 
For further study of the infinite speed of propagation and the so-called acceleration effect see \cite{FKT2016,FT2017c,Gar2011}.

The paper is organised as follows. In Section~\eqref{sec:tech}, we describe the properies of the semi-flow generated by the equation \eqref{eq:basic} and connect Weinberger's scheme \cite{Wei1982a} with Theorem~\ref{thm:trwall}. In~Section~\ref{sec:propindir}, we study the propagation of a solution to \eqref{eq:basic} in a fixed direction and prove Theorem~\ref{thm:decayoutsidedirectional}. In Subsection~\ref{subsec:convto0}, we find sufficient conditions that $\Tau_*$ is a compact and has a non-empty interior, and we prove the first two items of Theorem~\ref{thm:combi}. In Subsection~\ref{subsec:convtotheta}, we extend Weinberger's scheme from discrete to continuous time and prove (Proposition~\ref{prop:conv_to_theta_cont_time}) the convergence \eqref{eq:globalbelowresult} under additional assumption on the initial condition. Finally, using the hair-trigger effect proved early in \cite{FT2017a}, we get rid on the latter restriction and prove the third item of Theorem~\ref{thm:combi}.

\section{Technical tools}\label{sec:tech}

\subsection{Properties of semi-flow}

For any $t\geq0$ and $0\leq f\in E=L^\infty(\X)$, we define the continuous semi-flow (see \cite{FKT100-1} for details) as follows
\begin{equation}
  (Q_{t}f)(x):=u(x,t),\qquad \text{a.a. } x\in\X,\label{def:Q_T}
\end{equation}
where $u(x,t)$ is the solution to \eqref{eq:basic} with the initial condition $u(x,0)=f(x)$. 
\begin{theorem}[\!\!{{\cite[Theorem 1.5]{FKT100-1}; see also \cite[Proposition 5.4]{FT2017a}}}]\label{thm:Qholds}
Let  \eqref{as:beta}--\eqref{as:compar} hold. Let $(Q_t)_{t\geq0}$ be the semi-flow  \eqref{def:Q_T} on the cone $\{0\leq f\in E\}$. Then, for each $t>0$, $Q=Q_t$ satisfies the following properties:
\begin{enumerate}[label=\textnormal{(Q\arabic*)}]
    \item $Q$ maps each of sets $E^+_\theta$, $E^+_\theta\cap C_b(\X)$, $E^+_\theta\cap\Buc$ into itself; \label{eq:QBtheta_subset_Btheta}
    \item let $T_y$, $y\in\X$, be a translation operator, given by  \label{prop:QTy=TyQ}
      \begin{equation}\label{shiftoper}
        (T_y f)(x)=f(x-y), \quad x\in\X,
      \end{equation}
      then
      \begin{equation}
        (QT_{y}f)(x)=(T_{y}Qf)(x), \quad x,y\in\X,\ f\in E^+_\theta;\label{eq:QTy=TyQ}
      \end{equation}
    \item $Q0=0$, $Q\theta=\theta$, and $Q r>r$, for any constant $r\in(0,\theta)$; \label{prop:Ql_gr_l}
    \item if $f,g\in E^+_\theta$, $f \leq g $, then $Qf \leq Qg$; \label{prop:Q_preserves_order}
    \item if $f_n,f\in E^+_\theta$, $f_{n}\locun f$, then $(Qf_{n})(x)\to (Qf)(x)$ for (a.a.) $x\in\X$; \label{prop:Q_cont}
    \item if $d=1$, then $Q:\M\to\M$.\label{prop:Q_Mtheta}
  \end{enumerate}
\end{theorem}
  Here and below $\locun$ denotes the locally uniform convergence of functions on $\X$ (in other words, $f_n\1_\La$ converge to $f\1_\La$ in $E$, for each compact $\La\subset\X$), and $\M$ denotes the set of all decreasing and right-continuous functions $f:\R\to[0,\theta]$.

For each $0\leq T_1<T_2<\infty$, let $\x_{T_1,T_2}$ denote the Banach space of all continuous mappings from $[T_1,T_2]$ to $E$ with the norm 
\[
  \|u\|_{T_1,T_2}:=\sup_{t\in[T_1,T_2]}\|u(\cdot,t)\|_E.
\]
For any $T>0$, we set also $\x_T:=\x_{0,T}$ and consider the subset $\mathcal{U}_T\subset\x_T$ of all mappings which are continuously differentiable on $(0,T]$. Here and below, we consider the left derivative at $t=T$ only. 

The property \ref{prop:Q_preserves_order} gives the comparison principle for solutions to \eqref{eq:basic}. To formulate a more general result needed for the sequel, consider, for each $T>0$ and $u\in\mathcal{U}_T$, 
\begin{equation}\label{Foper}
  (\mathcal{F}u)(x,t):=\dfrac{\partial u}{\partial t}(x,t)-\varkappa^{+}(a^{+}*u)(x,t)+mu(x,t)+u(x,t) \bigl( Gu \bigl)(x,t)
\end{equation}
for all $t\in(0,T]$ and a.a.~$x\in\X$.
\begin{proposition}[\!\!{{\cite[Proposition 2.8]{FKT100-1}, cf.~\cite[Theorem~2.3]{FT2017a}}}]\label{prop:fullcomp}
  Let \eqref{as:beta}--\eqref{as:compar} hold. Let $T>0$ be fixed and $u_1,u_2\in\mathcal{U}_T$ be such that, for all $t\in(0,T]$, $x\in\X$,
    \begin{gather}
    (\mathcal{F}u_1)(x,t)\leq (\mathcal{F}u_2)(x,t),\label{eq:max_pr_BUC:ineq}\\
    0 \leq u_1(x,t)\leq\theta, \qquad 0 \leq u_2(x,t)\leq \theta,\notag\\
      0\leq u_{1}(x,0)\leq u_{2}(x,0)\leq \theta.\notag
    \end{gather}
    Then, for all $t\in[0,T]$, $x\in\X$,
  \begin{equation}\label{eq:comparineq}
    0\leq u_{1}(x,t)\leq u_{2}(x,t) \leq \theta.
  \end{equation}
\end{proposition}

We will need also a weaker form of \ref{prop:Q_cont} under weaker assumptions. 
\begin{proposition}\label{prop:stab_Linf_ae}
Let \eqref{as:beta}, \eqref{as:compar}
hold. Let $(Q_t)_{t\geq0}$ be the semi-flow  \eqref{def:Q_T} on the cone $\{0\leq f\in E\}$. Let $T>0$ be fixed.  Consider a sequence of functions $u_{n}\in\x_T$ which are solutions
to \eqref{eq:basic} with uniformly bounded initial conditions: $u_{n}(\cdot,0)\in E^+_\theta$, $n\in\N$. Let $u\in\x_T$ be a solution to \eqref{eq:basic} with initial condition $u(\cdot,0)$ such that
$u_{n}(x,0)\to u(x,0)$, for a.a.~$x\in\X$. Then
$u_{n}(x,t)\to u(x,t)$, for a.a.~$x\in\X$, uniformly in $t\in[0,T]$.
\end{proposition}
\begin{proof}
Clearly, $u_n(\cdot,0)\in  E^+_\theta $ implies $u(\cdot,0)\in E^+_\theta $. By~\ref{eq:QBtheta_subset_Btheta}, $u_n(\cdot,t), u(\cdot,t)\in E^+_\theta $, $n\in\N$, for any $t\geq0$. 
We define, for any $n\in\N$, 
\[
\overline{u}_{n}(x,0):=\max\left\{ u_{n}(x,0),u(x,0)\right\},\qquad \underline{u}_{n}(x,0):=\min\left\{ u_{n}(x,0),u(x,0)\right\}.
\]
Then, clearly, $0\le\underline{u}_n(x,0)\le u(x,0)\le\overline{u}_n(x,0)\le\theta$, $n\in\N$, a.a.~$x\in\X$. Hence the corresponding solutions $\overline{u}_{n}(x,t)$, $\underline{u}_{n}(x,t)$ to  \eqref{eq:basic} belongs to $ E^+_\theta $ as well. By~\ref{prop:Q_preserves_order}, one has $\underline{u}_{n}(x,t)\leq u(x,t)\leq\overline{u}_{n}(x,t)$, $n\in\N$, $t\in[0,T]$, a.a.~$x\in\X$. In the same way, one gets $\underline{u}_{n}(x,t)\leq u_n(x,t)\leq\overline{u}_{n}(x,t)$ a.e. on
$\X\times[0,T]$. Therefore, it is enough to prove that $\overline{u}_{n}$ and $\underline{u}_{n}$ converge a.e. to $u$ 

Prove that $\overline{u}_{n}(x,t)\to u(x,t)$ for a.a.~$x\in\X$ uniformly in $t\in[0,T]$. For any $n\in\N$, the function $h_{n}(\cdot,t)=\overline{u}_{n}(\cdot,t)-u(\cdot,t)\in E^+_\theta $, $t\geq0$, satisfies the equation
$\frac{\partial}{\partial t} h_{n}= P_n h_n$ with
$h_{n,0}(x):=h_n(x,0)=\overline{u}_{n}(x,0)-u(x,0)\geq0$, a.a.~$x\in\X$, where, for any $0\leq h\in\xt$,
\[
  P_{n}h := -mh+\varkappa^{+}(a^{+}*h)-\kn h(a^{-}*\overline{u}_{n}) -\kn u(a^{-}*h) - \kl h (u+\overline{u}_n).
\]
For any $u_n$ and $u$, $P_n$ is a bounded linear operator on $E$, therefore,
$h_{n}(x,t)=(e^{tP_{n}}h_{n,0})(x)$, a.a.~$x\in\X$, $t\in[0,T]$. Since $u\geq0$, one has that, for any $0\leq h\in\xt$,
$(P_nh)(x,t)\leq (Ph)(x,t)$, a.a.~$x\in\X$, $t\in[0,T]$, where a bounded linear operator $P$ is given on $E$ by
\[
  Ph:=\varkappa^{+}(a^{+}*h)-\kn u(a^{-}*h) - \kl u h.
\]
Next, the series expansions for $e^{tP_n}$ and $e^{tP}$ converge in the topology of norms of operator on the space $E$. Then, for any $n\in\N$, $t\in[0,T]$ and a.a.~$x\in\X$,
\begin{equation}
h_n(x,t)=(e^{tP_{n}}h_{n,0})(x)\leq(e^{TP}h_{n,0})(x)=\sum_{m=0}^{\infty}\dfrac{T^{m}}{m!}P^m h_{n,0}(x), \label{eq:loc_stab_BUC:suff_est}
\end{equation}
and, moreover, for any $\eps>0$ and a.a.~$x\in\X$, one can find $M=M(\eps,x)\in\N$, such that we get from \eqref{eq:loc_stab_BUC:suff_est} that, for $t\in[0,T]$ and a.a.~$x\in\X$,
\begin{equation}
h_n(x,t)\leq\sum\limits _{m=0}^{M}\dfrac{T^{m}}{m!}P^m h_{n,0}(x)+\eps \theta, \label{eq:loc_stab_BUC:final_est}
\end{equation}
as $h_{n,0}\in E^+_\theta $, $n\in\N$.
Finally, the assumptions of the statement yield that $h_{n,0}(x)\to 0$ for a.a.~$x\in\X$. Then, by \eqref{eq:loc_stab_BUC:final_est} and \cite[Lemma 2.2]{FKT100-1}, $h_n(x,t)\to 0$ for a.a.~$x\in\X$ uniformly in $t\in[0,T]$. Hence, $\overline{u}_{n}(x,t)\to u(x,t)$ for a.a.~$x\in\X$ uniformly on $[0,T]$. The convergence for $\underline{u}_{n}(x,t)$ may be proved by an analogy.
\end{proof}

\subsection{Around Weinberger's scheme}\label{subsec:Weinberger}

We will follow the abstract scheme proposed in \cite{Wei1982a}. Let \eqref{as:beta}--\eqref{as:compar} hold. We introduce the following notation, cf.~\ref{eq:QBtheta_subset_Btheta} of Theorem~\ref{thm:Qholds},
\begin{equation}\label{eq:defUtheta}
      \Utheta:=E^+_\theta\cap C_b(\X).
\end{equation}

Consider the set $N_\theta$ of all non-increasing functions $\varphi\in C(\R)$, such that
$\varphi(s)=0$, $s\geq0$, and
\begin{equation*}
\varphi(-\infty):=\lim _{s\to-\infty}\varphi(s)\in(0,\theta).\label{def:fy_Weinberger}
\end{equation*}
It is easily seen that $N_\theta\subset \Utheta$.

For arbitrary $s\in\R$, $c\in\R$, $\xi\in\S $,
we define the mapping $V_{s,c,\xi}:L^\infty(\R)\to E$ as follows
\begin{equation}\label{defofV}
  (V_{s,c,\xi}f)(x):=f(x\cdot \xi+s+c), \quad x\in\X.
\end{equation}
Fix an arbitrary $\varphi\in N_\theta$.
 For $T>0$, $c\in\R$, $\xi\in\S $, consider the mapping $R_{T,c,\xi}:\ L^{\infty}(\R)\to L^{\infty}(\R)$, given by
\begin{equation}
(R_{T,c,\xi}f)(s):=\max\bigl\{ \varphi(s),(Q_T (V_{s,c,\xi}f))(0)\bigr\},\quad s\in\R,\label{eq:iterop_by_Weinberger}
\end{equation}
where $Q_T$ is given by \eqref{def:Q_T}.
Consider now the following sequence of functions
\begin{equation}\label{fiteration}
  f_{n+1}(s):=(R_{T,c,\xi}f_n)(s),\quad f_0(s):=\varphi(s),\qquad s\in\R, n\in\N\cup\{0\}.
\end{equation}
By Theorem~\ref{thm:Qholds} and \cite[Lemma~5.1]{Wei1982a}, $\varphi\in \Utheta$ implies $f_n\in \Utheta$ and $f_{n+1}(s)\geq f_{n}(s)$, $s\in\R$, $n\in\N$; hence one can define the following limit
\begin{equation}
f_{T,c,\xi}(s):= \lim_{n\to\infty}f_n(s), \quad s\in\R.\label{eq:limit_func_Weinberger}
\end{equation}
Also, by \cite[Lemma~5.1]{Wei1982a}, for fixed $\xi\in\S $, $T>0$, $n\in\N$, the functions $f_n(s)$ and $f_{T,c,\xi}(s)$ are nonincreasing in $s$ and in $c$; moreover, $f_{T,c,\xi}(s)$ is a lower semicontinuous function of $s,c,\xi$, as a result, this function is continuous from the right in $s$ and in $c$. Note also, that $0\leq f_{T,c,\xi}\leq\theta$.
Then, for any $c,\xi$, one can define the limiting value
\[
f_{T,c,\xi}(\infty):=\lim_{s\to\infty}f_{T,c,\xi}(s).
\]
Next, for any $T>0$, $\xi\in\S$, we define
\begin{equation}\label{eq:cTupstar}
c_T^{*}(\xi):=\sup\{ c\mid f_{T,c,\xi}(\infty)=\theta\} \in\R\cup\{-\infty,\infty\},
\end{equation}
where, as usual, $\sup\emptyset:=-\infty$. By \cite[Propositions~5.1, 5.2]{Wei1982a}, one has
\begin{equation}\label{jumpfunc}
  f_{T,c,\xi}(\infty)=\begin{cases}
    \theta, & c<c_T^*(\xi),\\
    0, & c\geq c_T^*(\xi),
  \end{cases}
\end{equation}
cf.~also \cite[Lemma~5.5]{Wei1982a}; moreover, $c_T^*(\xi)$ is a lower semicontinuous function of~$\xi\in\S$. It is crucial that, by \cite[Lemma 5.4]{Wei1982a}, neither $f_{T,c,\xi}(\infty)$ nor $c_T^{*}(\xi)$ depends on the choice of $\varphi\in N_\theta$. Note that the monotonicity of $f_{T,c,\xi}(s)$ in $s$ and \eqref{jumpfunc} imply that, for $c<c_T^*(\xi)$, $f_{T,c,\xi}(s)=\theta$, $s\in\R$.

Define now the following set, cf. \eqref{eq:cTupstar},
\begin{equation}
    \Tau_{T,\xi}=\bigl\{ x\in\X \mid x\cdot\xi\leq c_T^{*}(\xi)\bigr\}, \quad \xi\in\S , T>0.\label{eq:TauTxi}
\end{equation}
Clearly, the set $\Tau_{T,\xi}$ is convex and closed. 

Recall that, cf. the Introduction, if, under \eqref{as:beta}--\eqref{as:compar}, for a fixed $\xi\in\S$, the assumption \eqref{as:expintdir} holds, then $c_*(\xi)$ is given by the first  item of Theorem~\ref{thm:trwall}. Otherwise, if \eqref{as:expintdir} fails, we set $c_*(\xi)=\infty$.
\begin{proposition}\label{cstarsareequal}
 Let \eqref{as:beta}--\eqref{as:compar} hold. Then, for any $\xi\in\S$,  $c_*(\xi)<\infty$ if and only if $c^*_T(\xi)<\infty$ for all $T>0$, and
 \begin{equation}\label{ct=tc}
 c^*_T(\xi)=Tc_*(\xi), \quad T>0.
 \end{equation}
 As a result, cf.~\eqref{eq:condonfrontdir}, \eqref{eq:TauTxi},
 \begin{equation}\label{taut=ttau1-dir}
  \Tau_{T,\xi}=T\Tau_{1,\xi}=T\Tau_*(\xi), \quad T>0.
\end{equation}
\end{proposition}
\begin{proof}
Let $T>0$ and $c^*_T(\xi)<\infty$. 
Take any $c\in\R$ with $cT\geq c^*_T(\xi)$. Then, by \eqref{jumpfunc},
$f_{T,cT,\xi}\not\equiv\theta$.
By \eqref{eq:iterop_by_Weinberger}, \eqref{fiteration}, one has
\begin{equation}\label{qns1}
  f_{n+1}(s)\geq (Q_T (V_{s,cT,\xi}f_n))(0), \quad s\in\R.
\end{equation}
Since $f_n(s)$ is nonincreasing in $s$, one gets, by \eqref{defofV}, that, for a fixed $x\in\X$, the function $(V_{s,cT,\xi}f_n)(x)$ is also nonincreasing in $s$.
Next, by \eqref{defofV}, \eqref{eq:limit_func_Weinberger} and Propositions \ref{prop:stab_Linf_ae},
\begin{equation}\label{convspec}
  (Q_T (V_{s,cT,\xi}f_n))(x)\to (Q_T (V_{s,cT,\xi}f_{T,cT,\xi}))(x), \text{  a.a. } x\in\X.
\end{equation}
Note that, by \eqref{defofV} and \cite[Proposition~3.3]{FKT100-1},
\begin{equation}\label{eqqqphi}
  (Q_T (V_{s,cT,\xi}f_{T,cT,\xi}))(x)=\phi (x\cdot\xi,T),
\end{equation}
where $\phi(\tau,t)$, $\tau\in\R$, $t\in\R_+$ solves 
\begin{equation}
  \begin{cases}
    \begin{aligned}
      \dfrac{\partial \phi}{\partial t}(s,t)&=\varkappa^{+}(\widecheck{a}^{+}*\phi)(s,t)-m\phi(s,t) -\kl \phi^2(s,t) \\&\quad
        -\kn\phi(s,t)(\widecheck{a}^{-}*\phi)(s,t), \qquad t>0, \ \mathrm{a.a.}\ s\in\R,
    \end{aligned}\\
    \phi(s,0)=\psi(s),\qquad \mathrm{a.a.}\ s\in\R.
  \end{cases}\label{eq:basic_one_dim}
\end{equation}
 with $\psi(\tau)=f_{T,cT,\xi}(\tau+s+cT)$ (note that $s$ is a parameter now, cf.~\eqref{eq:basic_one_dim}), and
\begin{equation}\label{apm1dim}
    \widecheck{a}^\pm (s):=\int_{\{\xi\}^\bot} a^\pm (s\xi+\eta) \, d\eta, \quad s\in\R,
\end{equation}
where $\{\xi\}^\bot:=\{x\in\X\mid x\cdot\xi=0\}$.

On the other hand, the evident equality
\[
(V_{s,cT,\xi}f_{T,cT,\xi})(x+\tau\xi)=f_{T,cT,\xi}(x\cdot\xi+\tau+s+cT), \quad \tau\in\R
\]
shows that the function $V_{s,cT,\xi}f_{T,cT,\xi}$ is a decreasing function on $\X$ along the $\xi\in\S $ as $f_{T,cT,\xi}$ is a decreasing function on $\R$. Then, by \cite[Proposition~2.7]{FKT100-1} and \eqref{eqqqphi}, the function $\X\ni x\mapsto\phi (x\cdot\xi,T)\in[0,\theta]$ is decreasing
along the $\xi$ as well, i.e. 
\[
  \phi (x\cdot\xi+\tau,T)=\phi ((x+\tau\xi)\cdot\xi,T)\leq \phi(x\cdot\xi,T), \quad \tau\geq0. 
\]
As a result, the function $\phi (s,T)$ is monotone (almost everywhere) in $s$. Since $f_{T,cT,\xi}(s)$ was continuous from the right in~$s$, one gets from \eqref{qns1}, \eqref{convspec}, that
\[
f_{T,cT,\xi}(s)\geq (\tilde{Q}_T f_{T,cT,\xi})(s+cT),
\]
where $\widetilde{Q}_{t}:L^{\infty}(\R)\to L^{\infty}(\R)$ is defined as follows: $\widetilde{Q}_t\psi(s)=\phi(s,t)$, $s\in\R$, where $\phi:\R\times\R_+\to[0,\theta]$ solves \eqref{eq:basic_one_dim} with $0\leq\psi\in L^{\infty}_{+}(\R)$. Since $f_{T,cT,\xi}\not\equiv\theta$, one has that, by  \cite[Theorem 5]{Yag2009} (cf.~the proof of \cite[Theorem 1.1]{FKT100-1}), there exists a traveling wave profile with the speed $c$. By Theorem~\ref{thm:trwall}, we have that $c\geq c_*(\xi)$, and hence $Tc_*(\xi)\leq c^*_T(\xi)<\infty$. 

Let now $T>0$ and $c_*(\xi)<\infty$. Take any $c\geq c_*(\xi)$ and consider, by Theorem~\ref{thm:trwall}, a traveling wave in a direction $\xi\in\S $, with a profile $\psi\in\M$ and the speed $c$. Then, by \eqref{defofV} and \eqref{eq:deftrw},
\[
(Q_T(V_{s,cT,\xi}\psi))(x)=\psi((x\cdot\xi-cT)+s+cT)=\psi(x\cdot\xi+s).
\]
Choose $\varphi\in N_\theta$ such that $\varphi(s)\leq \psi(s)$, $s\in\R$ (recall that all constructions are independent on the choice of $\varphi$). Then, one gets from \eqref{eq:iterop_by_Weinberger} and \ref{prop:Q_preserves_order} of Theorem~\ref{thm:Qholds},
that
\[
(R_{T,cT,\xi}\varphi)(s)\leq(R_{T,cT,\xi}\psi)(s)=\psi(s), \quad s\in\R.
\]
Then, by \eqref{fiteration} and \eqref{eq:limit_func_Weinberger}, $f_{T,cT,\xi}(s)\leq\psi(s)$, $s\in\R$, and thus \eqref{jumpfunc} implies $cT\geq c_T^*(\xi)$; as a result, $c^*_T(\xi)\leq Tc_*(\xi)<\infty$, that fulfills the statement.
\end{proof}

A developement of Weinberger's scheme crucial for the sequel is the so-called \emph{hair-trigger effect}. We have proved it for a generalisation of \eqref{eq:basic} in \cite{FT2017a}. It is straightforward to check, cf.~\cite[Subsection~2.1]{FKT100-1}, that, in our settings, the result can be read as follows. 
\begin{theorem}[{{cf.~\cite[Theorem 2.5]{FT2017a}}}]\label{thm:hair-trigger}
  Let the conditions \eqref{as:beta}--\eqref{as:nondegglobalmod} hold. Let $u_0\in E_\theta^+$ be such that there exist $x_0\in\X$, $\eta>0$, $r>0$, with $u_0\geq\eta$, for a.a.~$x\in B_r(x_0)$. Let $u\in\x_\infty$ be the corresponding classical solution to \eqref{eq:basic} on~$\R_+$.
Then, for $\m$ defined by \eqref{firstfullmoment} and any compact set $K\subset \X$,
  \begin{equation}\label{eq:htformula}
    \lim_{t\to\infty} \essinf_{x\in K} u(x+t\m,t)=\theta.
  \end{equation}
 \end{theorem} 
 In particular, if $\m=0\in\X$, then the solution to \eqref{eq:basic} converges to $\theta$ locally uniformly. Our main aim in the rest of the paper is to show that the zone where the solution to \eqref{eq:basic} becomes arbitrary close to $\theta$ (as time grows to $\infty$) can be chosen expanding to $\X$ linearly in time, cf. \eqref{convtotheta} below.  

\section{Long-time behavior in a direction}\label{sec:propindir}
In this Section, we are going to prove Theorem~\ref{thm:decayoutsidedirectional}. 
We start with the following simple observation. Let $0\leq u_0\in E$ be an initial condition to \eqref{eq:basic} and $u=u(x,t)\geq0$ be the corresponding solution. Then, by Duhamel's principle, $u(x,t)\leq w(x,t)$, $x\in\X$, $t>0$, where $w(x,t)$ is the solution to the linear equation 
\begin{equation}\label{eq:linear}
    \dfrac{\partial w}{\partial t}(x,t)=\varkappa^+ \int_{\X }a^+  (x-y)w(y,t)dy-m w(x,t)
\end{equation}
with the same initial condition $w(x,0)=u_0(x)$, $x\in\X$. We will find now an appropriate upper estimate for the solution to \eqref{eq:linear}. 

To this end, for any $\xi\in\S $ and $\la>0$, consider the following set of bounded functions on $\X$:
\begin{equation}\label{sefElambda}
E_{\la,\xi}(\X):=\bigl\{ f\in E \bigm| \|f\|_{\la,\xi}:=\esssup_{x\in\X} \lvert f(x)\rvert e^{\la  x\cdot\xi }<\infty \bigr\}.
\end{equation}
Evidently, for $f\in E$,
\[
\esssup\limits_{x\in\X} |f(x)| e^{\la  x\cdot\xi }<\infty \quad \text{if and only if}\quad
\esssup\limits_{x\cdot\xi\geq0} |f(x)| e^{\la  x\cdot\xi }<\infty,
\]
therefore,
\begin{equation*}
  E_{\la,\xi}(\X)\subset E_{\la',\xi}(\X), \quad \la>\la'>0, \ \xi\in\S .
\end{equation*}

\begin{proposition}\label{prop:expdecay}
Let $\xi\in\S $ and $\la>0$ be fixed and suppose that \eqref{as:expintdir} holds with  $\mu=\la$. Let $0\leq u_0\in E_{\la,\xi}(\X)$ and let $w=w(x,t)$ be the solution to \eqref{eq:linear} with the initial condition $w(x,0)=u_0(x)$, $x\in\X$. Then
\begin{equation}\label{expdecaying-lin}
  \| w(\cdot,t)\|_{\la,\xi}\leq \|u_0\|_{\la,\xi}e^{ p  t}, \quad t\geq0,
\end{equation}
where
\begin{equation}\label{bignu}
 p = p (\xi,\la)=\varkappa^+ \int_\X a^+(x) e^{\la x\cdot\xi }\,dx-m\in\R.
\end{equation}
\end{proposition}
\begin{proof}
First, we note that, for any $a\in L^1(\X)$, $f\in E_{\la,\xi}(\X)$
\begin{align}\notag
\bigl\lvert(a*f)(x)e^{\la  x \cdot\xi }\bigr\rvert &\leq
\int_\X|a(x-y)|e^{\la  (x-y) \cdot\xi }|f(y)|e^{\la  y \cdot\xi }\,dy\\&\leq \|f\|_{\la,\xi} \int_\X |a(y)|e^{\la  y \cdot\xi }\,dy.\label{elaest}
\end{align}
Applying \eqref{elaest} to $a=a^+\in L^1(\X)$ and $f=u_0\in E_{\la,\xi}(\X)$, and using the notation \eqref{aplusexpla}, we will get
\[
  \|a^+*u_0\|_{\la,\xi} \leq \A_\xi(\la) \|u_0\|_{\la,\xi}.
\]
Iteratively applying \eqref{elaest} to $a=a^+$ and $f=a^{+,*(n-1)}*u_0\in E_{\la,\xi}(\X) $, $n\geq2$, where $a^{+,*(n-1)}:=a^+*\ldots * a^+$ (the convolution is taken $n-2$ times), we obtain
\[
  \|a^{+,*n}*u_0\|_{\la,\xi} \leq \bigl(\A_\xi(\la)\bigr)^n \|u_0\|_{\la,\xi}.
\]
Since the operator in the right hand side of \eqref{eq:linear} is bounded in $E$, we have
an explicit representation for the solution to \eqref{eq:linear}, namely,
\[
  w(x,t)=e^{-mt}u_0(x)+e^{-mt}\sum_{n=1}^\infty \frac{(\varkappa^+ t)^n}{n!}\bigl(a^{+,*n}*u_0\bigr)(x), \quad x\in\X,\ t\geq0.
\]
As a result, we obtain
\[
  \| w(\cdot,t)\|_{\la,\xi}\leq e^{-mt}\| u_0\|_{\la,\xi}+e^{-mt}\sum_{n=1}^\infty \frac{(\varkappa^+ t)^n}{n!} \bigl(\A_\xi(\la)\bigr)^n \|u_0\|_{\la,\xi},
\]
that is just equivalent to \eqref{expdecaying-lin}--\eqref{bignu}. 
\end{proof}

\begin{remark}
  It is straightforward to check, cf.~\cite[Lemma 2.1]{FKT100-2}, that the statement of Proposition~\ref{prop:expdecay} remains true if \eqref{as:expintdir} holds for some $\mu>\la$, provided that we assume, additionally, \eqref{as:bdd}.
\end{remark}

We can prove now Theorem~\ref{thm:decayoutsidedirectional}.

\begin{proof}[Proof of Theorem~\ref{thm:decayoutsidedirectional}]
Let $p_*:=p(\xi,\la_*)$ be given by \eqref{bignu}. Let $w=w(x,t)$ be the solution to \eqref{eq:linear} with the initial condition $w(x,0)=u_0(x)$, $x\in\X$. By \eqref{expdecaying-lin}, \eqref{sefElambda}, one has
\begin{equation}\label{mainassasf}
  0\leq u(x,t)\leq w(x,t)\leq \lVert u_0\rVert_{\la_*,\xi} \exp\bigl\{ p_* t - \la_* x\cdot\xi\bigr\}, \quad \text{a.a. } x\in\X.
\end{equation}
Next, by \eqref{eq:TauTxi} and Proposition~\ref{cstarsareequal},  for any $t>0$ and for all $x\in\X\setminus t\Tauout_\xi$, one has $x\cdot\xi\geq t c^*_1(\xi)+t\delta=t c_*(\xi)+t\delta$. Then, by~\eqref{eq:cstar},
\begin{multline*}
\inf_{x\notin t\Tauout_\xi } (\la_* x\cdot\xi)\geq t\la_* c_*(\xi)+t\la_* \delta\\=t\Bigl(\varkappa^+ \int_\X a^+(x) e^{\la_* x\cdot\xi }\,dx-m\Bigr)+t\la_* \delta=tp_*+t\la_* \delta.
\end{multline*}
Therefore, \eqref{mainassasf} implies the statement.
\end{proof}

\begin{remark}\label{rem:ads1}
 The assumption $u_0\in E_{\la_*,\xi}(\X)$ is close, in some sense, to the weakest possible assumption on an initial condition $u_0\in\Ltheta$ for the equation \eqref{eq:basic} to have
 \begin{equation}\label{convtp0general}
    \lim_{t\to\infty}\esssup_{x\notin t\Tauout_{\xi}} u(x,t)=0,
 \end{equation}
 for an arbitrary open set $\Tauout_{\xi}\supset \Tau_{1,\xi}$, where $\Tau_{1,\xi}$ is defined by \eqref{eq:TauTxi}.
 Indeed, take any $\la_1,\la$ with $0<\la_1<\la<\la_*=\la_*(\xi)$. By~Theorem~\ref{thm:trwall}, there exists a traveling wave solution to \eqref{eq:basic} with a profile $\psi_1\in\M$ such that $\la_0(\psi_1)=\la_1$. By \cite[Theorem 1.3]{FKT100-2} (with $j=1$ as $\la_1<\la_*$) we have that $\psi_1(t)\sim De^{-\la_1 t}$, $t\to\infty$. It is easily seen that one can choose a function $\varphi\in\M\cap C(\R)$ such that there exist $p>0$, $T>0$, such that $\varphi(t)\geq\psi_1(t)$, $t\in\R$ and $\varphi(t)=pe^{-\la t}$, $t>T$. Take now $u_0(x)=\varphi(x\cdot\xi)$, $x\in\X$. We have $u_0\in E_{\la,\xi}(\X)\setminus E_{\la_*,\xi}(\X)$. Then, by \cite[Proposition~3.3]{FKT100-1}, the corresponding solution has the form $u(x,t)=\phi(x\cdot\xi,t)$. By Proposition~\ref{prop:fullcomp} applied to the equation \eqref{eq:basic_one_dim}, $\phi(s,t)\geq\psi_1(s-c_1t)$, $s\in\R$, $t\geq0$, where $c_1=\la_1^{-1}(\varkappa^+\A_\xi(\la_1)-m)>c_*(\xi)$, cf.~\cite[formula~(1.13)]{FKT100-2}. Take $c\in (c_*(\xi),c_1)$ and consider an open set $\Tauout_\xi: =\{x\in\X\mid x\cdot\xi<c\}$, then $\Tau_{1,\xi}\subset \Tauout_\xi \subset \{x\in\X\mid x\cdot\xi\leq c_1\}=:A_1$. One has
 \begin{align*}
 \sup_{x\notin t\Tauout_{\xi}}u(x,t)&\geq \sup_{x\in tA_1\setminus t\Tauout_{\xi}}\phi(x\cdot\xi,t)\\
 &\geq \sup_{ct<s\leq c_1t}\psi_1(s-c_1t)=\psi_1(ct-c_1t)>\psi_1(0),
 \end{align*}
 as $c<c_1$ and $\psi_1$ is decreasing. As a result, \eqref{convtp0general} does not hold.

 On the other hand, if $\psi_*\in\M$ is a profile with the minimal speed $c_*(\xi)\neq0$ and if $j=2$, cf.~\cite[Proposition~3.1]{FKT100-2}, then $u_0(x):=\psi_*(x\cdot\xi)$ does not belong to the space $E_{\la_*,\xi}(\X)$, and the arguments above do not contradict
\eqref{convtp0general} anymore. In the next remark, we consider this case in more details.
\end{remark}

\begin{remark}\label{rem:ads2}
  In connection with the previous remark, it is worth noting also that one can easily generalize Theorem~\ref{thm:decayoutsidedirectional} in the following way. Let $u_0\in E_{\la,\xi}(\X)\cap \Ltheta$, for some $\la\in (0,\la_*]$, and let $u\in\tXinf$ be the corresponding solution to \eqref{eq:basic}. Consider the set $A_{c,\xi}:=\bigl\{ x\in\X \mid x\cdot\xi\leq c\bigr\}$, where  $c=\la^{-1}(\varkappa^+\A_\xi(\la)-m)$
   cf.~\cite[formula~(1.12)]{FKT100-2}. Then, for any open set $B_{c,\xi}\supset A_{c,\xi}$ with $\delta_c:=\dist(A_{c,\xi},\X\setminus B_{c,\xi})>0$, one gets
\begin{equation}\label{equivasdas}
  \esssup_{x\notin t B_{c,\xi}} u(x,t)\leq \lVert u_0\rVert_{\la,\xi} e^{-\la\delta_c t}.
\end{equation}
Therefore, if $u_0(x)=\psi_*(x\cdot \xi)$, where $\psi_*$ is as in Remark~\ref{rem:ads1} above, then, evidently, $u_0\in E_{\la,\xi}(\X)$, for any $\la\in(0,\la_*)$. Then, for any open $\Tauout_\xi \supset\Tau_{1,\xi}$ with $\delta:=\dist (\Tau_{1,\xi},\X\setminus\Tauout_\xi )>0$  one can choose, for any $\eps\in(0,1)$, $c_1=c_*(\xi)+\delta\eps$. By Theorem~\ref{thm:trwall}, there exists a unique $\la_1=\la_1(\eps)\in(0,\la_*)$ such that $c_1=\la_1^{-1}(\varkappa^+\A_\xi(\la_1)-m)$. Then $u_0\in E_{\la_1,\xi}(\X)$ and
$A_{c_1,\xi}\subset \Tauout_\xi $, i.e. $\Tauout_\xi $ may be considered as a set $B_{c_1,\xi}$, cf.~above. As a result, \eqref{equivasdas} gives \eqref{supconvto0xi}, with the constant $\lVert u_0\rVert_{\la_1,\xi}<\lVert u_0\rVert_{\la_*,\xi}$, and with $\la_*\delta$ replaced by $\la_1 \delta(1-\eps)$. Note that, clearly, $\lVert u_0\rVert_{\la_1,\xi}\nearrow\lVert u_0\rVert_{\la_*,\xi}$, $\la_1\nearrow\la_*$, $\eps\to0$.
\end{remark}

\section{Long-time behavior in different directions}

\subsection{Convergence to $0$}\label{subsec:convto0}

Through this section we will assume that the conditions \eqref{as:beta}--\eqref{as:nondegglobalmod} hold. Let the convex closed set $\Tau_*$ be given by \eqref{eq:condonfrontglobal}. Define, cf. \eqref{eq:TauTxi},
\begin{equation}
\Tau_{T}=\left\{ x\in\X |x\cdot\xi\leq c_T^{*}(\xi),\ \xi\in\S \right\}, \quad T>0.\label{eq:TauT}
\end{equation}
By \eqref{eq:TauTxi}-\eqref{taut=ttau1-dir},
\begin{equation}\label{TauT=TTau1}
  \Tau_{T}=\bigcap_{\xi\in\S }\Tau_{T,\xi}=\bigcap_{\xi\in\S }T\Tau_{1,\xi}=T\Tau_1=T\Tau_*, \quad T>0;
\end{equation}
in particular, $\Tau_*=\Tau_1$.

\begin{proposition}\label{prop:verynew}
Let \eqref{as:beta}--\eqref{as:nondegglobalmod} hold. Then, cf.~\eqref{firstfullmoment}, $\m$ is an interior point of $\Tau_*$.
\end{proposition}
\begin{proof}
Firstly, if \eqref{as:expintdir} fails for all $\xi\in\S$ then $\Tau_*=\X$ and the statement is trivial. Next, for an arbitrary $\xi\in\S$ such that \eqref{as:expintdir} holds, we have, by \eqref{firstdirmoment} and the inequality in \eqref{eq:cstar}, that
\begin{equation}\label{dsasasa}
   \m\cdot\xi =\varkappa^+ \int_\X x\cdot \xi a^+(x)\,dx=\m_\xi<c_*(\xi).
\end{equation}
Therefore, cf.~\eqref{eq:condonfrontdir}, $\m\in \Tau_*(\xi)$, $\xi\in\S$. 
Next, as it was already mentioned, by \cite[Proposition~5.1]{Wei1982a}, the function $c_1^*(\xi)$ is lower-semicontinuous in $\xi\in\S$. Therefore, by \eqref{ct=tc}, the function 
$c_*(\xi)-\m_\xi>0$ is lower-semicontinuous on the compact $\S$, and hence attains its minimum, which we denote by $d_0>0$. As a result, $\m\cdot \xi <c_*(\xi)-d_0$ for all $\xi\in\S$, and therefore, an open ball with center at $\m$ and radius $d_0$ belongs to the interior of 
$\Tau_*(\xi)$, for each $\xi\in\S$. From this, by \eqref{eq:condonfrontglobal}, one gets the statement.
\end{proof}

\begin{proposition}\label{prop:front_is_non-empty}
Let \eqref{as:beta}--\eqref{as:expintglobal} hold. Then, $\Tau_*=\Tau_1$ is a compact.
\end{proposition}
\begin{proof}
First, \eqref{as:expintglobal} implies that  \eqref{as:expintdir} holds for all $\xi\in\S$. Then, by Theorem~\ref{thm:trwall}, $c_*(\xi)<\infty$ for all $\xi\in\S$.
Next, by~\eqref{firstdirmoment} and Proposition~\ref{prop:verynew}, for any orthonormal basis $\{e_i\mid 1\leq i \leq d\}\subset\S $, $\m=\sum\limits_{i=1}^d \m_{e_i}e_i\in\inter(\Tau_*)$. 
By Theorem~\ref{thm:trwall}, $x\in\Tau_*$ implies that, for any fixed $\xi\in\S $, $x\cdot \xi\leq c_*(\xi)$ and $x\cdot (-\xi)\leq c_*(-\xi)$, i.e.
\begin{equation}\label{dirbundforfront}
  -c_*(-\xi)\leq x\cdot\xi\leq c_*(\xi), \quad x\in\Tau_*, \ \xi\in\S .
\end{equation}
Then \eqref{dirbundforfront} implies
\[
\lvert x\cdot\xi\rvert\leq \max\bigl\{\lvert c_*(\xi)\rvert,\lvert c_*(-\xi)\rvert\bigr\}, \quad
x\in\Tau_*, \ \xi\in\S ;
\]
in particular, for an orthonormal basis $\{e_i\mid 1\leq i\leq d\}$ of $\X$, one gets
\[
|x|\leq \sum_{i=1}^d \lvert x\cdot e_i\rvert\leq \sum_{i=1}^d
\max\bigl\{\lvert c_*(e_i)\rvert,\lvert c_*(-e_i)\rvert\bigr\}=:R<\infty,\quad x\in\Tau_*,
\]
that fulfills the statement.
\end{proof}

\begin{remark}
Here and in Propositions~\ref{prop:combi-1}, \ref{prop:combi-2}, the condition \eqref{as:nondegglobalmod} can be weaken to \eqref{as:nondegglobal}. As a matter of fact, it is enough to assume that \eqref{as:nondegdir} holds for all $\xi\in\S$.
\end{remark}

\begin{remark}\label{rem:justequiv}
 Since $\int_{x\cdot \xi\leq0}a^+(x)e^{\la x\cdot\xi}\,dx\in[0,1]$, $\xi\in\S $, $\la>0$, we have the following observation. If, for some $\xi\in\S $, there exist $\mu^\pm>0$, such that, cf.~\eqref{aplusexpla}, $\A_{\pm\xi}(\mu^\pm)<\infty$, i.e. if \eqref{as:expintdir} holds for both $\xi$ and $-\xi$, then, for $\mu=\min\{\mu^+,\mu^-\}$,
 \begin{align}
&\quad \int_\X a^+(x) e^{\mu |x\cdot\xi|}\,dx=
 \int_{x\cdot\xi\geq0} a^+(x) e^{\mu x\cdot\xi}\,dx
 +\int_{x\cdot\xi<0} a^+(x) e^{-\mu x\cdot\xi}\,dx\notag
 \\&\leq \int_{x\cdot\xi\geq0} a^+(x) e^{\mu^+ x\cdot\xi}\,dx
 +\int_{x\cdot(-\xi)> 0} a^+(x) e^{\mu^- x\cdot(-\xi)}\,dx<\infty.
 \label{expmoddir}
\end{align}
Let now $\{e_i\mid 1\leq i\leq d\}$ be an orthonormal basis in $\X$. Let \eqref{as:expintdir} holds for $2d$ directions $\{\pm e_i\mid 1\leq i\leq d\}\subset\S $ and let $\mu_i=\min\{\mu(e_i),\mu(-e_i)\}$, $1\leq i\leq d$, cf.~\eqref{expmoddir}. Set $\mu=\frac{1}{d}\min\{\mu_i\mid 1\leq i\leq d\}$. Then, by the triangle and Jensen's inequalities and \eqref{expmoddir}, one has
\begin{align*}
      \int_\X a^+(x) e^{\mu|x|}\,dx
&\leq \int_\X a^+(x) \exp\biggl(\sum_{i=1}^d \frac{1}{d}\mu_i|x\cdot e_i|\biggr)\,dx\\
&\leq \sum_{i=1}^d \frac{1}{d} \int_\X a^+(x) e^{\mu_i|x\cdot e_i|}\,dx<\infty.
\end{align*}
Therefore, \eqref{as:expintglobal} is equivalent to that \eqref{as:expintdir} holds for all $\xi\in\S $. 
\end{remark}

\begin{remark}\label{rem:minspeedopospos}
  It is worth noting that, by \eqref{eq:cstar}, \eqref{firstdirmoment}, the following inequality holds, cf.~\eqref{dirbundforfront},
\begin{equation*}
  c_*(\xi)+c_*(-\xi)>\m_\xi+\m_{-\xi}=0.
\end{equation*}
\end{remark}

\medskip

We are ready to prove now the first item of Theorem~\ref{thm:combi}.
\begin{proposition}\label{prop:combi-1}
Let the conditions \eqref{as:beta}--\eqref{as:nondegglobalmod} hold and there exists $\xi\in\S$, such that \eqref{as:expintdir} holds. Let $u_0\in E_\theta^+$ be such that \eqref{eq:fastinitcond} holds for all those $\xi\in\S$ where $c_*(\xi)<\infty$. Let $u\in\x_\infty$ be the corresponding classical solution to \eqref{eq:basic} on $\R_+$. Then, for any compact set $\Tauin\subset\X\setminus\Tau_*$, there exist $\nu=\nu(\Tauin)>0$ and $D=D(u_0,\Tauin)>0$, such that
\begin{equation}\label{eq:globalaboveresult-prop}
  \esssup_{x\in t\Tauin} u(x,t)\leq D  e^{-\nu t}, \quad t>0.
\end{equation}
\end{proposition}
\begin{proof}
Since there exists $\xi\in\S$, such that \eqref{as:expintdir} holds, we will get from \eqref{eq:condonfrontglobal}, that $\Tau_*\neq\X$. Therefore,
\[
  \Tau_*=\bigcap_{\substack{\xi\in\S:\\ c_*(\xi)<\infty} }\bigl\{ x\in\X \mid x\cdot\xi\leq c_*(\xi)\bigr\}.
\]
Then a closed set $\Tauin\subset\X\setminus\Tau_*$ satisfies
\[
  \Tauin\subset \bigcup_{\substack{\xi\in\S:\\ c_*(\xi)<\infty} } \bigl\{ x\in\X \mid c_*(\xi)< x\cdot\xi\bigr\}.
\]
Since $\Tauin$ is a compact, there exist $K\in\N$ and $\xi_1,\ldots,\xi_K\in\S$, such that $c_*(\xi_i)<\infty$, $1\leq i\leq K$ and
\[
  \Tauin\subset \bigcup_{1\leq i\leq K} \bigl\{ x\in\X \mid x\cdot\xi_i> c_*(\xi_i)\bigr\}.
\]
Therefore,
\[
  \Tauout:=\X\setminus \Tauin\supset \bigcap_{1\leq i\leq K} \Tau_*(\xi_i).
\]
Clearly, $\Tauout$ is an open subset of $\X$ and $\Tauout\supset \Tau_*(\xi_i)$ for $1\leq i\leq K$. By the assumption on $u_0$ and the condition $c_*(\xi_i)<\infty$, $1\leq i\leq K$, the inequality \eqref{eq:fastinitcond} holds for all $\xi=\xi_i$, $1\leq i\leq K$. 

Since $\Tau_*(\xi_i)$ is a closed set and $\Tauin$ is a compact, we have that
 \[
  \nu_{i} :=\la_*(\xi_i) \, \dist (\Tau_*(\xi_i),\Tauin)>0, \quad 1\leq i\leq K.
\]
The inequality $c_*(\xi_i)<\infty$ implies that the condition \eqref{as:expintdir} holds for $\xi=\xi_i$, $1\leq i\leq K$. Therefore, by Theorem~\ref{thm:decayoutsidedirectional}, one gets, for any $1\leq i\leq K$,
\[
  \esssup_{x\in t\Tauin} u(x,t) =\esssup_{x\notin t\Tauout} u(x,t) \leq \lVert u_0\rVert_{\la_*(\xi_i),\xi_i} e^{-\nu_{i} t}\leq D e^{-\nu t}, \quad t>0,
\]
where $\nu:=\min\{\nu_{i}\mid 1\leq i\leq K\}$, $D:=\max\{\lVert u_0\rVert_{\la_*(\xi_i),\xi_i}\mid 1\leq i\leq K\}$.
\end{proof}

Prove now the second item of Theorem~\ref{thm:combi}.

\begin{proposition}\label{prop:combi-2}
In conditions and notations of Proposition~\ref{prop:combi-1}, we assume, additionally, that the set $\Tau_*$ is bounded (and hence compact). Then \eqref{eq:globalaboveresult} holds for any closed set $\Tauin\subset \X\setminus\Tau_*$.
\end{proposition}
\begin{proof}
Consider the set $\mathcal{M}$ of all subsets from $\X$ of the following form:
\begin{equation}
M=M_{\eps,K,\xi_1,\ldots,\xi_K}=\bigl\{x\in\X \mid
x\cdot\xi_i\leq c^*_1(\xi_i)+\eps,\ i=1,\ldots,K\bigr\},\label{eq:eps_neighb_ofGamma}
\end{equation}
for some $\eps>0$, $K\in\N$, $\xi_1,\ldots,\xi_K\in\S $.
By \eqref{TauT=TTau1} and Proposition~\ref{prop:verynew}, the set $\Tau_1=\Tau_*$ is bounded and nonempty. Take an arbitrary closed set $\Tauin\subset \X\setminus\Tau_*$, and consider the open set $\Tauout:=\X\setminus\Tauin\supset\Tau_*=\Tau_1$. Then, by \cite[Lemma~7.2]{Wei1982a}, there exist $\eps>0$, $K\in\N$, $\xi_1,\ldots,\xi_K\in\S $ and a set $M\in\mathcal{M}$ of the form \eqref{eq:eps_neighb_ofGamma}, such that
 \begin{equation}\label{maininclW}
   \Tau_*=\Tau_1\subset M\subset\Tauout.
\end{equation}
Choose now
\[
\Tauout_{\xi_i}=\Bigl\{x\in\X\Bigm\vert x\cdot\xi_i <c_1^*(\xi_i)+\frac{\eps}{2}\Bigr\}\supset \Tau_{1,\xi_i}, \quad 1\leq i\leq K.
\]
Then, by \eqref{maininclW}, 
\[
\Tau_*=\Tau_1=\bigcap_{\xi\in\S }\Tau_{1,\xi}\subset\bigcap_{i=1}^K \Tau_{1,\xi_i}\subset
\bigcap_{i=1}^K \Tauout_{\xi_i}\subset M\subset\Tauout,
\]
and, therefore,
\begin{equation}\label{complincl}
    \X\setminus \Tauout \subset \bigcup_{i=1}^K (\X\setminus\Tauout_{\xi_i}).
\end{equation}
Denote
 \[
  \nu_{i} :=\la_*(\xi_i) \, \dist (\Tau_{1,\xi_i},\X\setminus\Tauout_{\xi_i})=\la_*(\xi_i)\frac{\eps}{2}, \quad 1\leq i\leq K.
\]
Then, by Theorem~\ref{thm:decayoutsidedirectional} and \eqref{complincl}, one gets, for any $t>0$,
\[
  \esssup_{x\in t\Tauin} u(x,t)=\esssup_{x\notin t\Tauout} u(x,t) \leq \max_{1\leq i\leq K}\esssup_{x\notin t\Tauout_{\xi_i}} u(x,t)\leq D e^{-\nu t},
\]
with $\nu:=\min\{\nu_{i}\mid 1\leq i\leq K\}$, $D:=\max\{\lVert u_0\rVert_{\la_*(\xi_i),\xi_i}\mid 1\leq i\leq K\}$.
\end{proof}

\subsection{Convergence to $\theta$}\label{subsec:convtotheta}

We proof, at first, item 3 of Theorem~\ref{thm:combi} for uniformly continuous functions. Namely, we assume that $u_0\in\Utheta\cap\Buc$, $u_0\not\equiv0$, cf.~\eqref{eq:defUtheta}, and we will prove, under assumptions
\eqref{as:beta}--\eqref{as:nondegglobalmod}, that, for any compact set $\Tauin\subset\inter(\Tau_*)=\inter(\Tau_1)$,
\begin{equation}\label{convtotheta}
    \lim_{t\to\infty} \min_{x\in t\Tauin} u(x,t)=\theta.
  \end{equation}
To do this, in Proposition~\ref{prop:conv_to_theta_cont_time}, we apply results of \cite{Wei1982a} for discrete time, to prove \eqref{convtotheta} for continuous time, provided that $u_0$ is separated from $0$ on a large enough set. Then we will use the hair-trigger effect (Theorem~\ref{thm:hair-trigger}), which implies that $u(x,\tau)$ is separated from $0$ on an arbitrary large set (shifted by~$\tau\m$) for big enough $\tau>0$. Combining these results, we will get \eqref{convtotheta} for an arbitrary $u_0\in\Utheta\cap\Buc$, $u_0\not\equiv0$. Finally, by the comparison principle, we will get the third item  of Theorem~\ref{thm:combi} for $u_0\in\Ltheta$.

We start with the following Weinberger's result (rephrased in our settings).
Note that, under \eqref{as:beta}--\eqref{as:nondegglobalmod}, $\Tau_T\neq\emptyset$, $T>0$. Indeed, if there exists $\xi\in\S$, such that \eqref{as:expintdir} holds, then the result above follows from Proposition~\ref{prop:verynew} and \eqref{taut=ttau1-dir}. Otherwise, $\Tau_*=\X$ and \eqref{taut=ttau1-dir} yields the statement.

\begin{lemma}[{cf.~\cite[Theorem~6.2]{Wei1982a}}] \label{lem:conv_to_theta-1}
Let \eqref{as:beta}--\eqref{as:nondegglobalmod} hold.
Let $u_0\in \Utheta$ and $T>0$ be arbitrary, and $Q_T$ be given by \eqref{def:Q_T}  (in particular, $Q_T$ satisfies the properties \ref{eq:QBtheta_subset_Btheta}--\ref{prop:Q_cont} of Theorem~\ref{thm:Qholds}). Define
\begin{equation}\label{uniterarion}
u_{n+1}(x):=(Q_Tu_n)(x), \quad n\geq0.
\end{equation}
Then, for any compact set $\Tauin_T\subset\inter(\Tau_T)$ and for any $\sigma\in(0,\theta)$, one can choose a radius $r_\sigma=r_\sigma(Q_T,\Tauin_T)$, such that
\begin{equation}\label{initcondissepfrom0}
  u_0(x)\geq\sigma, \quad x\in B_{r_\sigma}(0),
\end{equation}
implies
\begin{equation}\label{liminfconv}
  \lim_{n\rightarrow\infty}\min\limits _{x\in n\Tauin_T}u_{n}(x)=\theta.
\end{equation}
\end{lemma}
\begin{remark}\label{justAradius}
By the proof of \cite[Theorem~6.2]{Wei1982a}, the radius $r_\sigma(Q_T,\Tauin_T)$ is not defined uniquely. In the sequel, $r_\sigma(Q_T,\Tauin_T)$ means just a~radius which fulfills the assertion of Lemma~\ref{lem:conv_to_theta-1} for the chosen $Q_T$ and $\Tauin_T$, rather than a~function of~$Q_T$~and $\Tauin_T$.
\end{remark}
\begin{remark}
It is worth noting, that, by \eqref{def:Q_T} and the uniqueness of the solution to \eqref{eq:basic}, the iteration \eqref{uniterarion} is just given by
\begin{equation}\label{realiteration}
  u_{n}(x)=u(x,nT), \quad x\in\X, n\in\N\cup\{0\}.
\end{equation}
Therefore, \eqref{liminfconv} with $T=1$ yields \eqref{convtotheta}, for $\N\ni t\to\infty$, namely,
\begin{equation}\label{convtothetanatural}
  \lim_{n\to\infty}\min_{x\in n\Tauin}u(x,n)=\theta,
\end{equation}
provided that \eqref{initcondissepfrom0} holds with $r_\sigma=r_\sigma(Q_1,\Tauin)$, $\Tauin\subset \inter(\Tau_1)$.
\end{remark}

\begin{lemma}
Let \eqref{as:beta}--\eqref{as:nondegglobalmod} hold. Fix a $\sigma\in(0,\theta)$ and a compact set $\Tauin\subset \inter(\Tau_1)$. Let $u_0\in\Utheta$ be such that $u_0(x)\geq\sigma$, $x\in B_{r_{\sigma}(Q_{1},\Tauin)}(0)$. Then, for any $k\in\N$,
\begin{equation}
\lim_{n\to\infty}\min_{x\in \frac{n}{k}\Tauin}u\Bigl(x,\dfrac{n}{k}\Bigr)=\theta.
\label{eq:rsigma1_geq_rsigma05:i}
\end{equation}
\end{lemma}
\begin{proof}
Since $\Tauin\subset\inter(\Tau_1)$, one can choose a compact
set $\tilde{\Tauin}\subset\inter(\Tau_1)$ such that
\begin{equation}\label{eq:rsigma1_geq_rsigma05:ii}
\Tauin\subset\inter(\tilde{\Tauin}).
\end{equation}
By \eqref{realiteration} and Lemma \ref{lem:conv_to_theta-1} (with $T=1$), the assumption
$u_0(x)\geq\sigma$, $x\in B_{r_{\sigma}(Q_{1},\Tauin)}(0)$ implies \eqref{convtothetanatural}. Fix $k\in\N$, take $p=\frac{1}{k}$; then choose and fix the radius $r_{\sigma}\bigl(Q_{p},p\tilde{\Tauin}\bigr)$. By~\eqref{convtothetanatural}, there exists an~$N=N(k)\in\N$, such that
\begin{gather*}
u(x,N)\geq\sigma,\quad x\in N\Tauin, \\
 B_{r_{\sigma}(Q_{p},p\tilde{\Tauin})}(0)\subset N\Tauin.
\end{gather*}
Apply now Lemma~\ref{lem:conv_to_theta-1}, with $u_0(x)=u(x,N)$, $x\in\X$, $T=p$, and
\[
\Tauin_T=\Tauin_p:=p\tilde{\Tauin}\subset p\,\inter(\Tau_1)=\inter(\Tau_p),
\]
as, by \eqref{TauT=TTau1}, $p\Tau_1=\Tau_p$. We will get then
\begin{equation}\label{eq:rsigma1_geq_rsigma05:iii}
\lim_{n\to\infty}\min_{x\in np\tilde{\Tauin}}u(x,N+np)=\theta.
\end{equation}
By (\ref{eq:rsigma1_geq_rsigma05:ii}), there exists $M\in\N$ such
that one has
\begin{equation}\label{inclwithn}
  \Bigl(\frac{N}{n}+p\Bigr)\Tauin\subset p\tilde{\Tauin},\quad n\geq M.
\end{equation}
Therefore, by \eqref{inclwithn}, one gets, for $n\geq M$,
\begin{align}\notag
\min_{x\in np\tilde{\Tauin}}u(x,N+np)&\leq
\min_{x\in n(\frac{N}{n}+p)\Tauin}u(x,N+np)\\&=
\min_{x\in(Nk+n)\frac{1}{k}\Tauin}u\Bigl(x,(Nk+n)\frac{1}{k}\Bigr)\leq\theta.\label{eq:rsigma1_geq_rsigma05:iv}
\end{align}
By \eqref{eq:rsigma1_geq_rsigma05:iii} and \eqref{eq:rsigma1_geq_rsigma05:iv}, one gets the statement.
\end{proof}

Now, one can prove \eqref{convtotheta}, under an assumption on the initial condition.
\begin{proposition}\label{prop:conv_to_theta_cont_time}
Let \eqref{as:beta}--\eqref{as:nondegglobalmod} hold. Fix a $\sigma\in(0,\theta)$ and a compact set $\Tauin\subset \inter(\Tau_1)$. Let $u_0\in\Utheta\cap\Buc$ be such that $u_0(x)\geq\sigma$, $x\in B_{r_{\sigma}(Q_{1},\Tauin)}(0)$, and $u\in\Xinf$ be the corresponding solution to \eqref{eq:basic}. Then \eqref{convtotheta} holds.
\end{proposition}
\begin{proof}
Suppose \eqref{convtotheta} were false.
Then, there exist $\eps>0$ and a sequence $t_N\to\infty$, such that
$\min\limits_{x\in t_N\Tauin}u(x,t_N)<\theta-\eps$, $n\in\N$.
Since $t_N\Tauin$ is a compact set and, by \ref{eq:QBtheta_subset_Btheta} in Theorem~\ref{thm:Qholds}, 
\begin{equation}\label{eq:newdop}
   u(\cdot,t)\in\Utheta\cap\Buc, \quad t\geq0,
\end{equation}
there exists $x_N\in t_N\Tauin$, such that
\begin{equation}\label{eq:conv_to_theta:i}
u(x_N,t_N)<\theta-\eps, \quad n\in\N.
\end{equation}
Next, by \eqref{eq:newdop} and~\cite[Proposition 5.1]{FT2017a}, there exists a $\delta=\delta(\eps)>0$
such that, for all $x',x''\in\X$ and for all $t', t''>0$, with $|x'-x''|+|t'-t''|<\delta$,
one has
\begin{equation}\label{eq:conv_to_theta:ii}
|u(x',t')-u(x'',t'')|<\dfrac{\eps}{2}.
\end{equation}
Since $\Tauin$ is a compact, $p(\Tauin):=\sup\limits_{x\in\Tauin}\lVert x\rVert<\infty$. Choose $k\in\N$, such that $\frac{1}{k}<\frac{\delta}{1+p(\Tauin)}$.
By~\eqref{eq:rsigma1_geq_rsigma05:i}, there exists $M(k)\in\N$, such that, for all $n\geq M(k)$, \begin{equation}\label{conseqofle}
  u\Bigl(x,\frac{n}{k}\Bigr)>\theta-\frac{\eps}{2}, \quad x\in \frac{n}{k}\Tauin.
\end{equation}
Choose $N>N_0$ big enough to ensure $t_N>\frac{M(k)}{k}$.
Then, there exists $n\geq M(k)$, such that $t_N\in\bigl[\frac{n}{k},\frac{n+1}{k}\bigr)$. Hence
\begin{equation}\label{sadas}
  \Bigl\lvert t_N-\frac{n}{k}\Bigr\rvert <\frac{1}{k}<\frac{\delta}{1+p(\Tauin)}.
\end{equation}
Next, for the chosen $N$, there exists $y_N\in\Tauin$, such that $x_N=t_N y_N$.
Set $t'=t_N$, $t''=\frac{n}{k}$, $x'=x_N=t_N y_N$, and $x''=\frac{n}{k}y_N$.
Then, by \eqref{sadas},
\[
|t'-t''|+|x'-x''|=\Bigl\lvert t_N-\frac{n}{k}\Bigr\rvert\bigl(1+|y_N|\bigr)<\delta.
\]
Therefore, one can apply \eqref{eq:conv_to_theta:ii}. Combining this with \eqref{eq:conv_to_theta:i}, one gets
\[
u\Bigl(\frac{n}{k}y_N,\frac{n}{k}\Bigr)=
u\Bigl(\frac{n}{k}y_N,\frac{n}{k}\Bigr)-u(t_N y_N,t_N)+
u(x_N,t_N)<\frac{\eps}{2}+\theta-\eps=\theta-\frac{\eps}{2},
\]
that contradicts \eqref{conseqofle}, as $\frac{n}{k} y_N\in\frac{n}{k}\Tauin$.
Hence the statement is proved.
\end{proof}

Now, we are ready to prove the third item of Theorem~\ref{thm:combi}.
\begin{proposition}\label{prop:combi-3}
Let the conditions \eqref{as:beta}--\eqref{as:nondegglobalmod} hold. Let $u_0\in E_\theta^+$ be such that there exist $x_0\in\X$, $\eta>0$, $r>0$, with $u_0(x)\geq\eta$ for a.a.~$x\in B_r(x_0)$; and let $u\in\x_\infty$ be the corresponding classical solution to \eqref{eq:basic} on $\R_+$.
Then, for any compact set $\Tauin\subset\inter(\Tau_*)$, the convergence \eqref{eq:globalbelowresult} holds.
\end{proposition}
\begin{proof}
At first, we suppose that $u_0\in\Utheta\cap\Buc$. 
For $u_0\equiv\theta$, the statement is trivial. Hence let $u_0\not\equiv\theta$, $u_0\not\equiv0$. Recall that, \eqref{as:nondegglobalmod} implies
  \eqref{as:nondegglobal}. 

Let $\Tauin\subset\inter(\Tau_1)$ be an arbitrary compact set. It is well-known, that the distance between disjoint compact and closed sets is positive; in particular, one can consider the compact $\Tauin$ and the closure of $\X\setminus\Tau_1$. Therefore, there exists a compact set $\K\subset\inter(\Tau_1)$, such that $\Tauin\subset\inter(\K)$. Let $\delta_0>0$ be the distance between $\Tauin$ and the closure of $\X\setminus\K$.

Choose any $\sigma\in(0,\theta)$ and consider a radius $r_\sigma=r_\sigma(Q_1,\K)$ which fulfills Proposition~\ref{prop:conv_to_theta_cont_time}, cf.~Remark~\ref{justAradius}. By~Theorem~\ref{thm:hair-trigger}, there exists $t_1>0$, such that
\begin{equation}\label{eq:dop}
    u(x+t_1\m,t_1)\geq \sigma, \quad |x|\leq r_\sigma.
\end{equation}
We apply now Proposition~\ref{prop:conv_to_theta_cont_time} (with $\Tauin$ replaced by $\K$) to the equation \eqref{eq:basic} with
\[
  u_0(x):=u(x+t_1\m,t_1), \quad x\in\X
\]
By~\eqref{convtotheta} and the uniqueness arguments, we will have then
\begin{equation}\label{convtothetaunderassumtion231}
  \lim_{t\to\infty}\min_{x\in t\K}u(x+t_1\m,t+t_1)=\theta.
\end{equation}

By \eqref{convtothetaunderassumtion231}, for any $\eps>0$, there exists $t_2>0$ such that, for all $t>t_1+t_2=:t_3>0$ and for all $y\in\K$,
\begin{equation}\label{sasasfsfwr}
 u\bigl((t-t_1)y+t_1\m,t\bigr)>\theta-\eps
\end{equation}
Without loss of generality we can assume that $t_2$ is big enough to ensure
\begin{equation}\label{bigt5}
  t_1\max\limits_{x\in\Tauin}|x|+t_1|\m|<\delta_0 t_2.
\end{equation}
Then, for any $x\in\Tauin$ and for any $t>t_3$, the vector
\[
y(x,t):=\frac{tx-t_1\m}{t-t_1}
\]
is such that
\[
\lvert y(x,t)-x\rvert
=\frac{\bigl\lvert t_1x-t_1\m\bigr\rvert}{t-t_1}
<\delta_0,
\]
where we used \eqref{bigt5}.
Therefore, $y(x,t)\in\K$, for all $x\in\Tauin$ and $t>t_3$, and hence \eqref{sasasfsfwr}, being applied for any such $y(x,t)$, yields
\[
  u(tx,t)>\theta-\eps, \quad x\in\Tauin, \ t>t_3, 
\]
that fulfills the proof of \eqref{convtotheta} for $u_0\in\Utheta\cap\Buc$. 

Let now $u_0\in\Ltheta$ satisfies the assumptions.
Then there exists a function $v_0\in\Utheta\cap\Buc\subset\Ltheta$, $v_0\not\equiv0$, such that $u_0(x)\geq v_0(x)$, for a.a.~$x\in\X$. Next, by Proposition~\ref{prop:fullcomp}, $u(x,t)\geq v(x,t)$, for a.a.~$x\in\X$, and for all $t\geq0$, where $v\in\x_\infty$ is the corresponding to $v_0$ solution to \eqref{eq:basic}. Then, by the proved above, we will get \eqref{convtotheta} for $v$, with the same $\Tau_1$, cf.~\ref{eq:QBtheta_subset_Btheta} of Theorem~\ref{thm:Qholds}. As a result, the evident inequality
  \[
  \min\limits_{x\in t\Tauin}v(x,t)\leq
  \essinf\limits_{x\in t\Tauin}u(x,t)\leq\theta
  \]
implies \eqref{eq:globalbelowresult}. The statement is fully proved now.
\end{proof}

Now one can prove Proposition~\ref{prop:infinitespeed}.
\begin{proof}[Proof of Proposition~\ref{prop:infinitespeed}]
The first statement is a direct consequence of the third item in Theorem~\ref{thm:combi}, since \eqref{eq:heavy} implies that $\Tau_*=\X$. To prove the second statement, suppose that, in contrast, for some $\xi\in\S$, $c\in\R$, and $\psi\in\M$, \eqref{eq:deftrw} holds. Then $u_0(x)=\psi(x\cdot\xi)$ satisfies the assumptions of the first statement. Take a compact set $\K\subset\X$, such that $c_1:=\max\limits_{y\in\K}y\cdot\xi >c$. Then \eqref{eq:globalbelowresult} implies
  \begin{align*}
    \theta&=\lim_{t\to\infty} \essinf_{x\in t\K} \psi(x\cdot\xi -ct)=\lim_{t\to\infty} \essinf_{y\in \K} \psi\bigl(t(y\cdot\xi -c)\bigr)\\
    &=
    \lim_{t\to\infty} \psi\bigl(t(c_1 -c)\bigr)=0,
  \end{align*}
  where we used that $\psi$ is decreasing. One gets a contradiction which proves the second statement.
\end{proof}

Another important application of the third item in Theorem~\ref{thm:combi} is that there are not stationary solutions $u\geq0$  to \eqref{eq:basic} (i.e. solutions with $\frac{\partial}{\partial t}u=0$), except $u\equiv0$ and $u\equiv\theta$, provided that the origin belongs to $\inter(\Tau_*)$.

\begin{proposition}\label{uniqstationarysolutions}
Let \eqref{as:beta}--\eqref{as:nondegglobalmod} hold. If $\kl=0$ in \eqref{eq:basic}, we assume, additionally, that there exists $r_0>0$ such that
  \begin{equation}\label{infpos}
    \alpha:=\inf\limits_{|x|\leq r_0} a^-(x)>0.
  \end{equation}
Let also the origin belong to $\inter(\Tau_*)$.
Then there exist only two non-negative stationary solutions to \eqref{eq:basic}
in $E$, namely, $u\equiv0$ and $u\equiv\theta$.
\end{proposition}
\begin{proof}
Since $\frac{\partial}{\partial t}u=0$, one gets from \eqref{eq:basic} that
\begin{equation}\label{stateq1}
  u(x)=\frac{\pm\sqrt{D(x)} - \bigl(m+B(x)\bigr)} {\kl}, \quad x\in\X,
\end{equation}
where
\begin{align*}
  A(x) &= \varkappa^+ (a^+*u)(x),\ B(x) = \kn(a^-*u)(x),\\
  D(x) &= \bigl( m+B(x) \bigr)^2 + 4\kl A(x) \geq m > 0.
\end{align*}
Then, by \cite[Lemma~2.1]{FKT100-1}, one easily gets that $u\in\Buc$.

Denote $M:=\|u\|=\sup\limits_{x\in\X}u(x)$. We are going to prove now that $M\leq\theta$. On the contrary, suppose that $M>\theta$.
One can rewrite \eqref{stateq1} as follows:
\begin{multline}
  mu(x) + \kl u^2(x) + \kn(a^-*u)(x)(u(x)-\theta)\\
  = (J_\theta*u)(x) \leq M(\varkappa^+{-}\kn\theta), \label{stateq2}
\end{multline}
where 
\[
  J_\theta(x):=\varkappa^+a^+(x)-\theta\kn a^-(x)\geq0,
\]
and hence $\int_\X J_\theta(x)\,dx=\varkappa^+{-}\kn\theta$.

Choose a sequence $x_n\in\X$, $n\in\N$, such that $u(x_n)\to M$, $n\to\infty$. Substitute $x_n$ to the inequality \eqref{stateq2} and pass $n\to\infty$. Since $M>\theta$ and $u\geq0$, one gets then that $(a^-*u)(x_n)\to0$, $n\to\infty$. Passing to a subsequence of $\{x_n\}$ and keeping the same notation, for simplicity, one gets that
\[
(a^-*u)(x_n)\leq\dfrac{1}{n},\ n\geq 1.
\]

For all $n\geq r_0^{-2d}$, set $r_n:=n^{-\frac{1}{2d}}\leq r_0$; then the inequality \eqref{infpos} holds, for any $x\in B_{r_n}(0)$, and hence
\begin{equation}\label{eq:st_sol_inf_est}
\dfrac{1}{n}\geq(a^-*u)(x_n)\geq\alpha(\1_{B_{r_n}(0)}*u)(x_n)\geq \alpha V_d(r_n)\min\limits_{x\in B_{r_n}(x_n)}u(x),
\end{equation}
where $V_d(R)$ is a volume of a sphere with the radius $R>0$ in $\X$. Since $V(r_n)=r_n^d V_d(1)=n^{-\frac{1}{2}}V_d(1)$, we have from \eqref{eq:st_sol_inf_est}, that, for any $n\geq r_0^{-2d}$, there exists $y_n\in B_{r_n}(x_n)$, such that
\[
u(y_n)\leq\dfrac{1}{\alpha\sqrt{n}V_d(1)}.
\]
Thus $u(y_n)\to0$, $n\to\infty$. Recall that $u(x_n)\to M>0$, $n\to\infty$, however, $|x_n-y_n|\leq r_n =n^{-\frac{1}{2d}}$, that may be arbitrary small. This contradicts the fact that $u\in\Buc$.

As a result, $0\leq u(x)\leq \theta=M$, $x\in\X$. Let $u\not\equiv0$. By the third item in Theorem~\ref{thm:combi}, for  any compact set $\Tauin\subset\inter(\Tau_1)$,
  $\min\limits_{x\in t\Tauin} u(x)\to\theta$, $t\to\infty$, as $u(x,t)=u(x)$ now. Since $0\in\inter(\Tau_1)$, the latter convergence is obviously possible for $u\equiv\theta$ only.
\end{proof}

\begin{remark}\label{rem:zero_is_in_front}
It is worth noting that, by \eqref{eq:TauTxi}, \eqref{taut=ttau1-dir}, and \eqref{ct=tc}, the assumption $0\in\inter(\Tau_1)$ implies that $c_*(\xi)\geq0$, for all $\xi\in\S$. It means that all traveling waves in all directions have nonnegative speeds only.
\end{remark}

\section*{Acknowledgments}
Authors gratefully acknowledge the financial support by the DFG through CRC 701 ``Stochastic
Dynamics: Mathematical Theory and Applications'' (DF, YK, PT), the European
Commission under the project STREVCOMS PIRSES-2013-612669 (DF, YK), and the ``Bielefeld Young Researchers'' Fund through the Funding Line Postdocs: ``Career Bridge Doctorate\,--\,Postdoc'' (PT).

\end{document}